\documentclass[10pt]{amsart}
\usepackage{amssymb,amsthm,amsmath,amsfonts}
\usepackage[numbers]{natbib}
\usepackage{hyperref}
\usepackage{enumerate}
\usepackage{soul}

\makeatletter
\g@addto@macro\th@plain{\thm@headpunct{}}
\makeatother
\usepackage{accents}

\usepackage{mathtools}

\makeatletter
\DeclareRobustCommand\widecheck[1]{{\mathpalette\@widecheck{#1}}}
\def\@widecheck#1#2{%
    \setbox\z@\hbox{\m@th$#1#2$}%
    \setbox\tw@\hbox{\m@th$#1%
       \widehat{%
          \vrule\@width\z@\@height\ht\z@
          \vrule\@height\z@\@width\wd\z@}$}%
    \dp\tw@-\ht\z@
    \@tempdima\ht\z@ \advance\@tempdima2\ht\tw@ \divide\@tempdima\thr@@
    \setbox\tw@\hbox{%
       \raise\@tempdima\hbox{\scalebox{1}[-1]{\lower\@tempdima\box
\tw@}}}%
    {\ooalign{\box\tw@ \cr \box\z@}}}
\makeatother
\newtheorem{theorem}{Theorem}[section]
\newtheorem{lemma}{Lemma}[section]
\newtheorem{proposition}{Proposition}[section]
\newtheorem{definition}{Definition}[section]
\newtheorem{corollary}{Corollary}[section]
\newtheorem{example}{Example}[section]
\newtheorem{remark}{Remark}[section]
\DeclareMathOperator{\R}{\mathbb{R}}
\DeclareMathOperator{\E}{\mathbb{E}}
\DeclareMathOperator{\V}{\mathbb {V}}
\DeclareMathOperator{\tr}{tr}

\DeclareMathOperator{\PP}{\mathbb {P}}
\DeclareMathOperator{\ZV}{\mathcal{Z}_{\mathcal{V}}}
\DeclareMathOperator{\ZG}{\mathcal{Z}_G}
\DeclareMathOperator{\CV}{\mathcal{P}_{\mathcal{V}}}
\DeclareMathOperator{\CVa}{\mathcal{Q}_{\mathcal{V}}}
\DeclareMathOperator{\VV}{\mathcal {V}}
\DeclareMathOperator{\ZVa}{\mathcal{Z}^\ast_{\mathcal{V}}}
\DeclareMathOperator{\RV}{\mathcal{R}_{\underline{s}}}
\DeclareMathOperator{\RVa}{\mathcal{R}^\ast_{\underline{s}}}
\DeclareMathOperator{\vc}{vec}

\def \kA{\mathcal{A}}

\providecommand{\scalar}[1]{\left\langle#1\right\rangle}
\providecommand{\set}[2]{ \left\{\,#1 \colon #2\,\right\} }
\newcommand{\nc}{\newcommand}
\nc{\us}{\underline{s}}
\nc{\un}{\underline{n}}
\nc{\HV}{H_{\mathcal{V}}}
\nc{\GL}{GL}

\title[Wishart laws and variance function on homogeneous cones]
{ Wishart {laws} and variance function on homogeneous cones}

\author{Piotr Graczyk}
\address{Laboratoire de Math\'ematiques LAREMA, Universit\'e d'Angers,2 Boulevard Lavoisier, 49045 Angers Cedex 01, France}
 \email{graczyk@univ-angers.fr}
  
\author{Hideyuki Ishi}
\address{Graduate School of Mathematics, Nagoya University, Furo-cho, Nagoya 464-8602, Japan\newline
\phantom{JS} JST, PRESTO, 4-1-8, Honcho, Kawaguchi 332-0012, Japan}
  \email{hideyuki@math.nagoya-u.ac.jp} 
    
\author{Bartosz Ko\l{}odziejek}
\address{Faculty of Mathematics and Information Science, Warsaw University of Technology,
  Koszykowa 75, 00-662 Warsaw, Poland}
  \email{b.kolodziejek@mini.pw.edu.pl}

\keywords{natural exponential families; variance function; Wishart laws; Riesz measure; homogeneous cones; graphical cones}

\begin{document}

\begin{abstract}
We present a systematic study of
Riesz measures and their natural exponential families of Wishart laws on a 
homogeneous cone. We compute explicitly the inverse of the mean map 
and the variance function of a Wishart exponential family.
\end{abstract}

\maketitle

\section{Introduction}
Modern statistics and multivariate analysis require use of models on
	subcones of the cone $\mathrm{Sym}_+(n,\R)$ of positive definite
	symmetric matrices. Such subcones are obtained for example by
	prescribing some of the off-diagonal elements to be $0$. One considers
	two types of Wishart laws on such subcones \cite{LM07}. First Type
	corresponds to the law of the Maximum Likelihood estimators of
	covariance matrices in a sample of size $n$ of a multidimensional
	normal vector $(X_1, \ldots, X_p)$, subject to conditional
	independence constraints, \cite{Lbook,LM07}. The inverses of the
	second Type Wishart forms a conjugate family of priors for the
	covariance parameter of the graphical Gaussian model. Some sub-classes
	of second Type Wishart laws, called $G$-Wishart in \cite{LM07}, are
	Diaconis-Ylvisaker conjugate priors (\cite{D-Y}) for the precision
	matrix.

Many of such subcones are homogeneous,  i.e. their automorphism group acts transitively. 
Recall that in the theory of graphical models \cite{Lbook}, a decomposable graph generates a homogeneous cone if and only if 
graph induced by any quadruple
of vertices is not the graph $\bullet-\bullet-\bullet-\bullet$, denoted by $A_4$, cf. \cite{Is13, LM07}. 	
Among all decomposable graphs with four
vertices, approximately 80\% are homogeneous. 

	 	 Consider the case when the vector $(X_1,\ldots,X_p)$ may be partitioned into independent subvectors. We impose no additional conditional independence constraints between elements of a given block. In such case, the covariance matrix as well as the precision matrix has block diagonal decomposition. Such model corresponds to a homogeneous cone and is often considered in statistics of sparse graphical models, 
	 	see the book of Hastie, Tibshirani and Wainwright	\cite[Section 9.3.3]{HastieTW}. 
	
	 	 Sparse graphical models have been intensely studied for over 15 years. However, the important problem
	 	of choosing thresholds for detection of dependence
	 	(equivalently, for the identification of non-zero terms in the precision matrix) is still open.
	 	The knowledge of Wishart laws on these models may be very useful in solving this problem. 
	 
	  Let us underline one more statistical motivation for {developing} the  theory of Wishart laws on homogeneous cones.
	 	Modern Big Data statistics concentrates on the case $n\ll p$
	 	in the multivariate normal sample of the vector $(X_1,\ldots,  X_p)$, cf. \cite{HastieTW}.
	 	This case corresponds to singular Wishart laws, which at the moment are fully described only on homogeneous cones.
	 	 cf. Theorems \ref{riesz} and \ref{hyper} below. 	

	 Consequently, statistical graphical models are often based on homogeneous graphs {(decomposable and not containing $A_4$ as an induced subgraph)}. This fact is regrettably 
	still ignored by most of the statistical community. 
	
	Moreover, important recent statistical articles \cite{LM07,KBala}  point out the significance    
	 of homogeneous cones among cones corresponding to decomposable and DAG graphical models and devote 
	 much space to multivariate analysis on homogeneous cones.  Further, Wishart laws on homogeneous cones supply many good
	 	examples of exponential families
	 	with very explicit calculation. This should be emphasized as a "raison
	 	d'\^etre" of our research.

The preponderant role of homogeneous cones among subcones of $\mathrm{Sym}_+(n,\R)$
strongly motivates research on Wishart laws on general homogeneous cones.
Families of Wishart laws on homogeneous cones were   
first studied by Andersson and Wojnar \cite{A-W}, Boutouria \cite{Bou} and
Letac and Massam \cite{LM07}. {The pioneering paper \cite{A-W}
is very technical {and inaccessible}, due  to the use of methods  based on Vinberg algebras.}  
 The natural approach to this topic, based on quadratic maps and matrix realizations of
homogeneous cones, was proposed in Graczyk and Ishi \cite{GI14,Is14}
and used in Ishi and Ko{\l}odziejek \cite{IK15}. 

This article is a continuation of \cite{GI14,Is14,IK15}, 
 however here we identify the dual space using the trace inner product instead of the standard inner product.
Indeed, the trace inner product
 will be indispensable in future applications to statistics.
Though the difference between standard and trace inner products is rather technical,
 it gives rise to some changes in formulas in our previous works
 about generalized power 
 functions and Wishart exponential families, i.e. the natural exponential families generated by Riesz measures.
Thus we repeat and simplify
 some definitions and proofs from \cite{GI14} and \cite{Is14}
 for completeness and convenience of the reader.
We hope that, eventually,
 this paper presents our methods and results in a manner which is accessible to mathematical statisticians.

The variance function  is an important {characteristic}
of a natural exponential family, cf. \cite{BN14}.
Let us consider a classical Wishart exponential family 
on the cone $\mathrm{Sym}_+(n,\R)$, i.e. the natural exponential family 
generated by 
 a {Riesz} measure $\mu_{{p}}$ {on $\mathrm{Sym}_+(n,\R)$, } with  
 the Laplace transform
$L_{\mu_{p}}(\theta)=(\det\theta)^{-{{p}}}$, for 
$\theta\in \mathrm{Sym}_+(n,\R)$.
Here $p$ belongs to the Gindikin-Wallach set (in this context also called the {J}\o rgensen set)
$\Lambda=\{1/2,1,3/2,\ldots,(n-1)/2\}
\cup((n-1)/2,\infty)$. It is well known and straightforward to check (see for example \cite{Le89}) that 
the variance function of NEF generated by $\mu_p$ is given by
\begin{align}\label{first}
\V_{{p}}(m)=\frac{1}{{{p}}}\rho(m)
\end{align}
where $
 m\in {(\mathrm{Sym}_+(n,\R))^\ast\equiv}
\mathrm{Sym}_+(n,\R)$ 
 and $\rho(m)$ is a linear map from $\mathrm{Sym}(n,\R)$ to itself defined by
 $\rho(m)Y=m Y m^\top$, $Y\in{\mathrm{Sym}}(n,\R)$.
Generally, Wishart exponential family is the natural exponential family generated by some Riesz measure.
In this paper we  give  explicit formulas for the variance function of Wishart exponential family on any homogeneous cone.

Let us describe shortly the plan of this paper. In Section \ref{nef} we recall 
the definition
of a natural exponential family generated by a positive measure 
and we introduce their characteristics, mean  and variance, used and studied throughout the whole paper.

Sections \ref{HCones}
and \ref{Riesz} are devoted to introducing the main tools for the analysis of Wishart
exponential families on  homogeneous cones.
Like in \cite{GI14}, we consider two types of homogeneous cones,
 $\CV$ which is in a matrix realization and its dual cone $\CVa$,
 and Riesz measures and Wishart families on them. 
This corresponds to the concept of Type I and Type II Wishart laws defined and studied in \cite{LM07}. 
 
Any homogeneous cone is linearly isomorphic to some matrix cone $\CV$ \cite{Is15}. 
It was observed in \cite{GI14} that the matrix {realization} of a homogeneous cone makes analysis of Riesz and Wishart measures much easier.
{The present paper is also based on this technique, which is reviewed in Section \ref{Matrix}.} 

In Section  \ref{gener}, we define
the generalized power functions  $\delta_{\us}$ and $\Delta_{\us}$ on the cones $\CVa$ and $\CV$, respectively. 
 In Proposition  \ref{proPhi}(iii), 
we give a formula  
 (\ref{New_delta})
 for the power function $\delta_{\us}$ , which  is new and useful.
In Definition \ref{hat} we introduce 
 an important map $\xi \rightarrow \hat \xi$ between $\CVa$ and
$\mathrm{Sym}_+(N,\R)$,
{inspired by an analogous map playing a fundamental role
 for decomposable graphical models \cite{Lbook}.}

In Section \ref{SectionInvMean}, 
 we first prove Formula \eqref{psi}, which serves as a simple and useful tool in our argument.
Then we deduce from \eqref{psi} an explicit evaluation of the inverse of the mean map ({Theorem}  \ref{InvMean}).
It allows us to find the Lauritzen formula on any (not only graphical) homogeneous cone.
We get in Section \ref{VarQ} the variance function formula   
for Wishart exponential families defined on the  cone $\CVa$. 
The proof is based on Formula \eqref{psi} again.
{Note that   Riesz and  Wishart measures on cones $Q_G$ and $P_G$ 
 corresponding  to the decomposable graphs $G$ were studied 
 in \cite{LM07} and (uniquely in case $Q_G$) in  
 \cite{A-K}. 
When the cones $Q_G$ and $P_G$ are homogeneous, the integral formulas in   \cite[Th. 3.1 and 3.2]{LM07} 
and  \cite[Th. 5.1]{A-K}   are a special case of our results     
on the cones $\CVa$ and $\CV$, when the dimensions of all the blocks
in the matrix realization of $\CVa$ are equal to 1.}

In Section \ref{VarP} we give results on  the variance function formula   
for Wishart exponential families on $\CV$. In particular, we {offer}
a practical approach to the construction of a matrix realization of 
the dual cone $\CVa$, using basic quadratic maps. 
Note that if a homogeneous cone $\CV$ is in matrix realization, then, in general, $\CVa$ is not, which makes analysis on $\CVa$ harder.

The last Section \ref{appli} contains applications of the results and of the 
methods of this paper to symmetric cones and graphical homogeneous cones.
We improve and complete in this way the results 
of \cite{HaLa01} and \cite{LM07}.

 In the Appendix we indicate why the results of an unpublished paper
\cite{BH09} are false or unproven  and cannot be compared with the results of our paper. The unpublished paper
\cite{BH09} should not be used by statisticians as a reference, that regrettably happens. 
The errors and deep gaps of \cite{BH09} illustrate well how difficult and fine the analysis
on homogeneous cones is. 

{Let us end this Introduction by some remarks on  terminology.
	We call the elements of
	the natural exponential family 
	generated by 
	a Riesz measure {\it Wishart laws (measures, {distributions})} .
   This terminology is classical on the symmetric cone $\mathrm{Sym}_+(n,\R)$
	and 
	is used for cones corresponding to graphical models in the Lauritzen monography \cite{Lbook}
	and in the fundamental papers \cite{LM07, KBala}. However, 
	the elements of
	the natural exponential family 
	generated by 
	a Riesz measure are called  {\it  ``Riesz probability (or generalized) measures''}
	in
	\cite{A-K,A-W,Bou,HaLa01}. We choose the terminology based on the  classical case of the cone $\mathrm{Sym}_+(n,\R)$ and on the
	basic references \cite{Lbook,LM07,KBala}. Moreover, the notion of
	a {\it Riesz distribution} comes from functional analysis (cf. Faraut-Koranyi book \cite{FaKo1994}), in particular a classical Riesz measure is never finite.}
	
	{ The elements of a natural exponential family 
	generated by 
	a Riesz measure are Wishart laws, so instead of
	saying ``natural exponential family 
	generated by 
	a Riesz measure'' or
	``Riesz exponential family'', we say {\it Wishart exponential family}. Note that any Wishart measure generates the exponential family to which it belongs, which is one more reason to speak about
	 {\it Wishart exponential families.}	}\\
	 	 
	{\bf Acknowledgement.} We thank Ma\l gorzata Bogdan and H\'el\`ene Massam for discussions on statistical applications of Wishart laws on homogeneous cones.

{This research benefited from the support of the French government ``Investissements d'Avenir'' program ANR-11-LABX-0020-01 and PDLL Regional grant D\'efiMaths.
H. Ishi was partially supported by JSPS KAKENHI Grant Number 16K05174 and JST PRESTO.
B. Ko{\l}odziejek was partially supported by NCN Grant No. 2012/05/B/ST1/0055.}

\section{Natural Exponential Families}\label{nef}

In the following section we will give a short introduction to natural exponential families (NEFs). The standard reference book on exponential families is \cite{BN14}.

Let $\E$ be a finite dimensional real linear space endowed with an inner product $\scalar{\cdot,\cdot}$ and let $\E^\ast$ be the dual space of $\E$. If $\xi\in\E^\ast$ is a linear functional on $\E$, we will denote its action on $x\in\E$ by $\scalar{\xi,x}$. Let $L_S(\E^\ast,\E)$ be the linear space of linear operators $A\colon \E^\ast\to\E$ such that
 for any $\xi, \eta\in\E^\ast$, one has $\scalar{{\xi},A({\eta})}=\scalar{{\eta},A({\xi})}$. 

Let $\mu$ be a positive Radon measure on $\E$. We define its Laplace transform $L_\mu\colon \E^\ast\to(0,\infty]$ by
$$L_\mu(\theta):=\int_{\E} e^{-\scalar{\theta,x}}\mu(dx).$$
Let $\Theta(\mu)$ denote the interior of the set $\{\theta\in \E^\ast\colon L_\mu(\theta)<\infty\}$. H\"older's inequality implies that the set $\Theta(\mu)$ is convex and the cumulant function
$$k_\mu(\theta):=\log L_\mu(\theta)$$
is convex on $\Theta(\mu)$ and it is strictly convex if and only if $\mu$ is not concentrated on {any} affine hyperplane of $\E$.
Let $\mathcal{M}(\E)$ be the set of positive Radon measures on $\E$ such that $\Theta(\mu)$ is not empty and $\mu$ is not concentrated on {any} affine hyperplane of $\E$. 

For $\mu\in\mathcal{M}(\E)$ we define the \emph{natural exponential family (NEF) generated by $\mu$} as the set of probability measures 
$$F(\mu)=\{P(\theta,\mu)(dx)=e^{-\scalar{\theta,x}-k_\mu(\theta)}\mu(dx)\colon \theta\in\Theta(\mu)\}.$$
Then, for $\theta\in \Theta(\mu)$,
\begin{align*}
m_\mu(\theta)&:=-k_\mu'(\theta)=\int_{\E}x\,P(\theta,\mu)(dx),\\
-m_\mu'(\theta)&=k_\mu''(\theta)=\int_{\E} (x-m_{\mu}(\theta))\otimes (x-m_\mu(\theta))\,P(\theta,\mu)(dx),
\end{align*}
are respectively the mean and the covariance operator of the measure $P(\theta,\mu)$. 
Here $x\otimes x$ is {an element of $L_S(\E^\ast, \E)$}
 defined by $(x\otimes x)({\xi})=x\scalar{{\xi},x}$ for $x\in \E$ and ${\xi}\in\E^\ast$.
The subset $M_{F(\mu)}:=m_\mu(\Theta(\mu))$ of $\E$ is called the domain of means of $F(\mu)$. The map $m_\mu\colon \Theta(\mu)\to M_{F(\mu)}$ is an analytic diffeomorphism, and its inverse is denoted by $\psi_\mu\colon M_{F(\mu)}\to \Theta(\mu)$.
\begin{lemma}[\mbox{\cite[Proposition IV.4]{Is14}}] \label{lem:Jmu}
Define $J_{\mu}(m) := \sup_{\theta \in \Theta(\mu)} \frac{e^{- \scalar{\theta, m}}}{L_{\mu}(\theta)}$ for any $m \in M_{F(\mu)}$.
Then $\psi_{\mu} = -(\log J_{\mu})'$. 
\end{lemma}
\begin{proof}
Since $\log J_{\mu}(-m) = \sup_{\theta \in \Theta(\mu)} (\scalar{\theta, m} - k_{\mu}(\theta))$ is the Legendre-Fenchel transform of $k_{\mu}(\theta)$,
 the statement follows from the Fenchel duality.
\end{proof}
For any $m\in M_{F(\mu)}$ consider the covariance operator $\V_{F(\mu)}(m)$ of the measure $P(\psi_\mu(m),\mu)$. Then
\begin{align}\label{defV}
	\V_{F(\mu)}(m)=k_\mu''(\psi_\mu(m))=-[\psi_\mu'(m)]^{-1}.
\end{align}
The map $\V_{F(\mu)}\colon M_{F(\mu)}\to L_S(\E^\ast,\E)$ is called \emph{the variance function} of $F(\mu)$. {The} variance function is {a} central object of interest of natural exponential families, because it characterizes a NEF and the generating measure in the following way: if $F(\mu)$ and $F(\mu_0)$ are two natural exponential families such that $\V_{F(\mu)}$ and $\V_{F(\mu_0)}$ coincide on a non-void open set $J\subset M_{F(\mu)}\cap M_{F(\mu_0)}$, then $F(\mu) = F(\mu_0)$ and so $\mu_0(dx)=\exp\{\scalar{a,x}+b\}\mu(dx)$ for some $a\in\E^\ast$ and $b\in\R$. The variance function gives full knowledge of the NEF.

In the context of natural exponential families, often invariance properties {under} the action of a subgroup of general linear group or general affine group are considered. For the recent developments in this direction see \cite{IK15}.

Usually when one defines natural exponential family one starts with the moment generating function $L_\mu(\theta)=\int_{\E} \exp\scalar{\theta,x}\mu(dx)$. In such case, introducing the same concept of {inverse of the }mean map as above, the covariance operator has the form $\V_{F(\mu)}(m)=[\psi_\mu'(m)]^{-1}$. For our purposes however we find it more convenient to define NEF through the Laplace transform.%
\section{Basic facts on homogeneous cones}\label{HCones}
Let $\mathrm{Mat}(n,m;\R)$, $\mathrm{Sym}(n,\R)$  denote the linear spaces of real $n\times m$ matrices and symmetric real $n\times n$ matrices,  respectively.
Let 
 $\mathrm{Sym}_+(n,\R)$ be the cone of symmetric positive definite real $n\times n$ matrices.
 $A^\top$ denotes the transpose of a matrix $A$. 
 For $X\in \mathrm{Mat}(N,N;\R)$ define a linear operator
 $\rho(X)\colon \mathrm{Sym}(N,\R)\to \mathrm{Sym}(N,\R)$ by 
 $$\rho(X)Y=X Y X^\top,\qquad Y\in\mathrm{Sym}(N,\R).$$

Let $V$ be a real linear space and $\Omega$ a regular open
convex cone in $V$. Open convex cone $\Omega$ is \emph{regular} if
$\overline{\Omega}\cap(-\overline{\Omega})=\{0\}$. The 
linear automorphism group preserving the cone
is denoted by $\mathrm{G}(\Omega)=\{g\in \GL(V)\colon g\,\Omega
=\Omega\}$. The cone $\Omega$ is said to be \emph{homogeneous} 
if $\mathrm{G}(\Omega)$ acts transitively on $\Omega$.

\subsection{Homogeneous cones  $\CV$ and $\CVa$}\label{Matrix}
We recall from \cite{GI14} a useful realization of any homogeneous cone. Let us take a partition $N=n_1+\ldots+n_r$ of a positive integer $N$, and consider a system of vector spaces $\VV_{lk}\subset\mathrm{Mat}(n_l,n_k;\R)$, \mbox{$1\leq k< l\leq r$}, satisfying {the} following three conditions:
\begin{itemize}
  \item[{\rm (V1)}] $A\in\VV_{lk}$, $B\in\VV_{ki}$ $\implies$ $AB\in\VV_{li}$ for any $1\leq i<k<l\leq r$,
	\item[{\rm (V2)}] $A\in\VV_{li}$, $B\in\VV_{ki}$ $\implies$ $AB^\top\in\VV_{lk}$ for any $1\leq i<k<l\leq r$,
	\item[{\rm (V3)}] $A\in\VV_{lk}$ $\implies$ $AA^\top\in\R I_{n_l}$ for any $1\leq k<l\leq r$.
\end{itemize}
Let $\ZV$ be the subspace of $\mathrm{Sym}(N,\R)$ defined by
$$\ZV:=
\begin{Bmatrix}
x=\begin{pmatrix}
X_{11} & X_{21}^\top & \cdots & X_{r1}^\top \\
X_{21} & X_{22} & & X_{r2}^\top \\
\vdots & & \ddots & \\
X_{r1} & X_{r2} & & X_{rr}
\end{pmatrix}
\colon
\begin{array}{l}
X_{lk}\in \mathcal{V}_{lk},\,\, 1\leq k< l\leq r  \\
{X_{ll}=x_{ll}I_{n_l},\,\, x_{ll} \in \R,\,\,1\leq l\leq r} 
\end{array}
\end{Bmatrix}.
$$
We set 
$$\CV:=\ZV\cap\mathrm{Sym}_+(N,\R).$$
Then $\CV$ is a regular open convex cone in the linear space $\ZV$. Let $\HV$ be the group of real lower triangular matrices with positive diagonals defined by
$$\HV:=
\begin{Bmatrix}
T=\begin{pmatrix}T_{11} & & & \\ T_{21} & T_{22} & & \\ \vdots & & \ddots & \\ T_{r1} & T_{r2} & & T_{rr} \end{pmatrix}
\colon 
\begin{array}{l}
T_{lk}\in\VV_{lk},\,\, 1\leq k< l\leq r \\
T_{ll}=t_{ll}I_{n_l},\,\,t_{ll}>0,\,\, 1\leq l\leq r
\end{array}
\end{Bmatrix}.
$$
If $T\in\HV$ and $x\in\ZV$, then $\rho(T)x=T\,x\,T^\top\in\ZV$ thanks to $\mathrm{(V1)-(V3)}$. Moreover, $\rho(\HV)$ acts on the cone $\CV$ (simply) transitively (\cite[Proposition 3.2]{Is06}), that is, $\CV$ is a homogeneous cone. Our interest in $\CV$ is motivated by the fact that any homogeneous cone is linearly isomorphic to $\CV$ due to \cite[Theorem D]{Is06}.

Condition $\mathrm{(V3)}$ allows us to define an inner product $(\cdot|\cdot)$ on $\mathcal{V}_{lk}$, \mbox{$1\leq k<l\leq r$}, by
$$(AB^\top+BA^\top)/2=(A|B)I_{n_l},\qquad A,B\in\mathcal{V}_{lk}.$$
We define \emph{the trace inner product} on $\ZV$ by
$$\scalar{x,y}=\tr(xy)=\sum_{k=1}^r n_kx_{kk}y_{kk}+2\sum_{1\leq k<l\leq r} n_l(X_{lk}|Y_{lk}),\qquad x,y\in\mathcal{Z}_\mathcal{V}.$$
Using the trace inner product we identify the dual space $\ZVa$ with $\ZV$. 
Define the dual cone $\CVa$ by
$$\CVa:=\{\xi\in \ZV\colon \scalar{\xi,x}>0\,\,\forall x\in\overline{\CV}\setminus\{0\}\},$$
where $\overline{\CV}$ is the closure of $\CV$.
The dual cone $\CVa$ is also homogeneous. It is easily seen that $I_N\in\CVa$.

For $T \in \HV$, we denote by $\rho^\ast(T)$ the adjoint operator of $\rho(T) \in \GL(\ZV)$
 defined in such a way that $\scalar{\xi,\rho(T)x}=\scalar{\rho^\ast(T)\xi,x}$ for any $\xi,x\in\ZV$.
For any $\xi\in\CVa$ there exists a unique $T\in\HV$ such that $\xi=\rho^\ast(T)I_N$ (\cite[Chapter 1, Proposition 9]{Vi63}). 
\subsection{Generalized power functions}\label{gener}
Define a one-dimensional representation $\chi_{\us}$ of the triangular group $\HV$ by
$$\chi_{\us}(T):=\prod_{k=1}^r t_{kk}^{2 s_k},$$ 
where $\us=(s_1,\ldots,s_r)\in\mathbb{C}^r$. Note that any one-dimensional representation $\chi$ of $\HV$ is of the form $\chi_{\us}$ for some $\us\in\mathbb{C}^r$. 
\begin{definition}
Let $\Delta_{\us}\colon\CV\to \mathbb{C}$ be the function given by
$$\Delta_{\us}(\rho(T)I_N):=\chi_{\us}(T),\qquad T\in \HV.$$
Let $\delta_{\us}\colon\CVa\to \mathbb{C}$ be the function given by
$$\delta_{\us}(\rho^\ast(T)I_N):=\chi_{\us}(T),\qquad T\in \HV.$$
\end{definition}
Functions $\Delta$ and $\delta$ are called \emph{generalized power functions}.

Let $N_k=n_1+\ldots+n_k$, $k=1,\ldots,r$. For $y\in\mathrm{Sym}(N,\R)$, by $y_{\{1:k\}}\in\mathrm{Sym}(N_k,\R)$ we denote the submatrix $(y_{ij})_{1\leq i,j\leq N_k}$. It is known that for any lower triangular matrix $T$ one has
$$(TT^\top)_{\{1:k\}}=T_{\{1:k\}}T_{\{1:k\}}^\top.$$
Thus, for $x=\rho(T)I_N\in\CV$ with $T\in\HV$ one has $\det x_{\{1:k\}}=(\det T_{\{1:k\}})^2=\prod_{i=1}^k t_{ii}^{2 n_i}$.
This implies that for any $x\in\CV$,
\begin{equation} \label{eqn:Delta_s} 
 \Delta_{\us}(x)=(\det x)^{\frac{s_r}{n_r}}\prod_{k=1}^{r-1}(\det x_{\{1:k\}})^{\frac{s_k}{n_k}-\frac{s_{k+1}}{n_{k+1}}}.
\end{equation}
We will express $\delta_{\us}(\xi)$ as a function of $\xi\in\CVa$ in the next Section (see Proposition \ref{proPhi}).

By definition, $\Delta_{\us}$ and $\delta_{\us}$ are multiplicative in the following sense
\begin{align}
\Delta_{\us}(\rho(T)x)&=\Delta_{\us}(\rho(T)I_N)\,\Delta_{\us}(x), \qquad (x,T)\in\CV\times\HV, \label{BigDel}\\
\delta_{\us}(\rho^\ast(T)\xi) &= \delta_{\us}(\rho^\ast(T)I_N)\,\delta_{\us}(\xi),\qquad\,\, (\xi,T)\in \CVa\times \HV. \label{SmallDel}
\end{align}

\begin{definition}\label{piDef}
Let $\pi\colon \mathrm{Sym}(N;\R)\to\ZV$ be the projection such that, for any $x\in\mathrm{Sym}(N,\R)$ the element $\pi(x)\in\ZV$ is uniquely determined by 
$$\tr(xa)=\scalar{\pi(x),a},\qquad \forall a\in\ZV.$$
\end{definition}

For any $x,y\in\ZV$ one has
$$\scalar{\rho^\ast(T)x,y}=\scalar{x,\rho(T)y}=\tr(\rho(T^\top)x\cdot y)=\scalar{\pi(\rho(T^\top)x),y},$$
thus, for any $T\in\HV$,
\begin{align}\label{rhos}
\rho^\ast(T)=\pi\circ \rho(T^\top).
\end{align}

Now we define a useful map $\xi \rightarrow \hat \xi$ between $\CVa$ and
$\mathrm{Sym}_+(N,\R)$,  such that $(\hat{\xi})^{-1}\in \CV$ and
$\pi(\hat{\xi})=\xi$. 
{An analogous map}
 is very important in statistics on decomposable graphical models \cite{Lbook}.

\begin{definition}\label{hat}
For $\xi=\rho^\ast(T)I_N\in\CVa$ with $T\in\HV$, we define 
$$\hat{\xi}:=\rho(T^\top)I_N=T^\top T\in\mathrm{Sym}_+(N,\R).$$
\end{definition}
Note that for any $\xi\in\CVa$, one has $(\hat{\xi})^{-1}\in \CV$ (compare the definition of $\hat{\xi}$ in \cite[Proposition 2.1]{LM07}). Indeed, $(\hat{\xi})^{-1}=\rho(T^{-1})I_N\in\CV$. 
Due to \eqref{rhos}, we have $\pi(\hat{\xi})=\xi$. 

Observe that for $T\in\HV$,
$$\Delta_{\us}(\rho(T)I_N)=\chi_{\us}(T)=\chi_{-\us}(T^{-1})=\delta_{-\us}(\rho^\ast(T^{-1})I_N)$$
and, due to \eqref{rhos}, $\rho^\ast(T^{-1})I_N=\pi\left((\rho(T)I_N)^{-1}\right)$.
This implies that functions $\Delta$ and $\delta$ are related by the 
following identity
\begin{equation}\label{Delta_delta}
\Delta_{\us}(x)=\delta_{-\us}(\pi(x^{-1})),\qquad x\in \CV,
\end{equation}
or equivalently,
\begin{equation} \label{eqn:Delta_delta2}
\Delta_{\us}(\hat{\xi}^{-1})=\delta_{-\us}(\xi),\qquad \xi\in\CVa.
\end{equation}
In literature, function $\delta_{\us}$ is sometimes denoted by $\Delta_{\us^\ast}^\ast$, where $\us^\ast=(s_r,\ldots,s_1)$.

\subsection{Basic quadratic maps {$q_i$} and associated maps $\phi_i$}\label{QUADR}
We recall from \cite{GI14} {a} construction of basic quadratic maps.
Let $W_i$, $i=1,\ldots,r$, be the subspace of $\mathrm{Mat}(N,n_i;\R)$ consisting of {the} matrices $x$ of the form
$$x=\begin{pmatrix}
0_{n_1+\ldots+n_{i-1},n_i} \\ x_{ii}I_{n_i} \\ \vdots \\X_{ri}
\end{pmatrix},$$
where $X_{li}\in\VV_{li}$, $l=i+1,\ldots, r$.
{For $x \in W_i$, the symmetric matrix $x x^\top$ belongs to $\ZV$ thanks to (V2) and (V3).}
We define \emph{the basic quadratic map} $q_i\colon W_i\ni x\mapsto x x^\top\in \ZV$.

Taking an orthonormal basis of each $\VV_{li}$ with respect to $(\cdot|\cdot)$, we identify the space 
$W_i$ with $\R^{m_i}$, where $m_i=\dim W_i=1+\dim\VV_{{i+1,i}}+\ldots+ \dim\VV_{ri}$. {Let $\vc(x)\in\R^{m_i}$ denote the vectorization of $x\in W_i$.}
It is convenient to choose a basis for $W_i$ consistent with the block decomposition of $\ZV$, that is,
$(v_1,\ldots,v_{m_i})$, where $v_1$ corresponds to $\VV_{ii}\simeq \R$ and $(v_2,\ldots,v_{1+\dim \VV_{{i+1,i}}})$ corresponds to $\VV_{i+1,i}$ and so on.\\
\begin{definition} \label{defi:phi_i}
For the quadratic map $q_i$ we define the associated linear map $\phi_i\colon \ZV\equiv\ZVa\to \mathrm{Sym}(m_i,\R)$ in such a way that for $\xi\in\ZV$,
$$\vc(x)^\top \phi_i(\xi)\vc(x)=\scalar{\xi,q_i(x)},\qquad \forall\,x\in W_i.$$
\end{definition}

Similarly we consider another subspace of $\mathrm{Mat}(N,n_i;\R)$, namely,
$$\widecheck{W}_i=\begin{Bmatrix}
x=\begin{pmatrix}0_{n_1+\ldots+n_{i},n_i} \\ X_{i+1,i} \\ \vdots \\X_{ri}\end{pmatrix}
\colon
X_{li}\in\VV_{li},\,\, l=i+1,\ldots,r
\end{Bmatrix},
$$
the quadratic map $\widecheck{q}_i\colon \widecheck{W}_i\ni x\mapsto x x^\top\in \ZV$ and its associated linear map {$\widecheck{\phi}_i\colon \ZV\equiv\ZVa\to \mathrm{Sym}(m_i-1,\R)$}.
\begin{proposition}\label{proPhi}
\begin{itemize}
\item[\rm (i)] For any $\xi\in\ZV$ and $i=1,\ldots,r-1$, one has
$$\phi_i(\xi)=
\begin{pmatrix} 
n_i\xi_{ii} & v_i(\xi)^\top\\
v_i(\xi) & \widecheck{\phi}_i(\xi)
\end{pmatrix},
$$
where $v_i(\xi):=\begin{pmatrix} n_{i+1}\vc(\xi_{i+1,i}) \\ \vdots \\ n_r\vc(\xi_{ri})\end{pmatrix}\in\R^{m_i-1}$ and $\vc(\xi_{ki})\in\R^{\dim\VV_{ki}}$ is the vectorization of $\xi_{ki}\in\VV_{ki}$.
Moreover, $\phi_r(\xi)= { n_r \xi_{rr} \in \R \equiv \mathrm{Sym}(1,\R)}$.

\item[\rm (ii)]
For $\xi=\rho^\ast(T)I_N\in\CVa$ with $T\in\HV$ and $i=1,\ldots,r-1$ one has
$$\det \phi_i(\xi)=\chi_{\underline{m}_i}(T)\det \phi_i(I_N),$$
and 
$$\det \widecheck{\phi}_i(\xi)=\chi_{\underline{\widecheck{m}}_i}(T)\det \widecheck{\phi}_i(I_N),$$
where 
\begin{align*}
\underline{m}_i:=(0,\ldots,0,1,n_{i+1,i},\ldots,n_{ri})\in\mathbb{Z}^r, \\
\underline{\widecheck{m}}_i:=(0,\ldots,0,0,n_{i+1,i},\ldots,n_{ri})\in\mathbb{Z}^r.
\end{align*}

\item[\rm (iii)] 
For any $\xi\in\CVa$, one has
\begin{equation}\label{New_delta}
\delta_{\us}(\xi)=C_{\us} \phi_r(\xi)^{{s_r}}\prod_{i=1}^{r-1} 
\left(\frac{\det \phi_i(\xi)}{\det \widecheck{\phi}_i(\xi)}\right)^{s_i},
\end{equation}
where the constant $C_{\us}$ does not depend on $\xi$.

\item[\rm (iv)]
For $\alpha,\beta\in\ZV$ and $i=1,\ldots,r-1$, one has
\begin{align*}
\tr \phi_i(\alpha)\phi_i(I_N)^{-1}\phi_i(\beta)&\phi_i(I_N)^{-1}-\tr\widecheck{\phi}_i(\alpha)\widecheck{\phi}_i(I_N)^{-1}\widecheck{\phi}_i(\beta)\widecheck{\phi}_i(I_N)^{-1} \\
&= \alpha_{ii}\beta_{ii}+2\sum_{l=i+1}^r \frac{n_l}{n_i}
(\alpha_{li}|\beta_{li}).
\end{align*}
\end{itemize}
\end{proposition}
Let us underline that the useful formula (\ref{New_delta}) for the power function $\delta_{\us}$  is new and different from {the} formula given in {\cite{GI14} and \cite{Is14}}.
Precisely, it is just mentioned in \cite{GI14, Is14} that $\delta_{\us}(\xi)$ is a product of powers of $\det \phi_i(\xi)$. 
\begin{proof}
\begin{itemize}
\item[\rm (i)] It is a consequence of the choice of basis for $W_i$ and the fact that $W_i\simeq \R\oplus \widecheck{W}_i$.

\item[\rm (ii)] In \cite{GI14} a very similar problem was considered, but there the dual space $\ZVa$ was identified with $\ZV$ using the, so-called, standard inner product, not the trace inner product. The only difference in the form of $\phi_i$ in these two cases is that here block sizes $n_i$ appear in $(i,i)$ component and in the definition of $v_i$. The proof is virtually the same for both cases - see \cite[Proposition 3.3]{GI14}. 

\item[\rm (iii)] From \rm{(ii)} we see that if $\xi=\rho^\ast(T)I_N$, then 
$$t_{ii}^2=\frac{\chi_{\underline{m}_i}(T)}{\chi_{\underline{\widecheck{m}}_i}(T)}=\frac{\det \widecheck{\phi}_i(I_N)}{\det \phi_i(I_N)}\frac{\det \phi_i(\xi)}{\det \widecheck{\phi}_i(\xi)}=n_i^{-1}\frac{\det \phi_i(\xi)}{\det \widecheck{\phi}_i(\xi)}.$$

\item[\rm (iv)] Using the block decomposition given in \rm{(i)}, one has
\begin{multline*}
\tr \phi_i(\alpha)\phi_i(I_N)^{-1}\phi_i(\beta)\phi_i(I_N)^{-1}-\tr\widecheck{\phi}_i(\alpha)\widecheck{\phi}_i(I_N)^{-1}\widecheck{\phi}_i(\beta)\widecheck{\phi}_i(I_N)^{-1} \\
 = \alpha_{ii}\beta_{ii}+
 { n_i^{-1} \tr \{ v_i(\alpha)^\top \widecheck{\phi}_i(I_N)^{-1} v_i(\beta)
           + v_i(\beta)^\top \widecheck{\phi}_i(I_N)^{-1} v_i(\alpha) \} }
\end{multline*}
and the assertion follows from the definition of $(\cdot|\cdot)$ and $v_i$.
\end{itemize}
\end{proof}

\section{Riesz measures and Wishart exponential families}\label{Riesz}
Generalized power functions play a very important role and this is due to the following
\begin{theorem}[\citep{Gi75, Is00}]\label{riesz}
\begin{itemize}
\item[\rm (i)]
	There exists a positive measure $\RV$ on $\ZV$ with the Laplace transform 
$$L_{\RV}(\xi)=\delta_{-\us}(\xi),\qquad \xi\in\CVa$$ 
if and only if $\us\in \Xi:=\bigsqcup_{\underline{\varepsilon}\in\{0,1\}^r} \Xi(\underline{\varepsilon})$ {(disjoint union)},
	where
	$$\Xi(\underline{\varepsilon}):=\begin{Bmatrix}
 & s_k>\frac12\sum_{i<k} \varepsilon_i\dim\VV_{ki}\mbox{ if } \varepsilon_k=1 \\
\underline{s}\in\R^r; & \\
	& s_k=\frac12\sum_{i<k} \varepsilon_i\dim\VV_{ki}\mbox{ if } \varepsilon_k=0
	\end{Bmatrix}.$$
The support of $\RV$ is contained in $\overline{\CV}$.
\item[\rm (ii)]
	There exists a positive measure $\RVa$ on $\ZVa\equiv\ZV$ with the Laplace transform
\begin{equation} \label{eqn:LT-Rs}
 L_{\RVa}(\theta)=\Delta_{-\us}(\theta),\qquad \theta\in\CV
\end{equation}
if and only if $\us\in \mathfrak{X}:={\bigsqcup}_{\underline{\varepsilon}\in\{0,1\}^r} \mathfrak{X}(\underline{\varepsilon})$,
	where
	$$\mathfrak{X}(\underline{\varepsilon}):=\begin{Bmatrix}
 & s_k>\frac12 \sum_{l>k}\epsilon_l\dim\VV_{lk}\mbox{ if } \varepsilon_k=1 \\
\us\in\R^r; & \\
	& s_k=\frac12 \sum_{l>k}\epsilon_l\dim\VV_{lk}\mbox{ if } \varepsilon_k=0
	\end{Bmatrix}.$$
The support of $\RVa$ is contained in $\overline{\CVa}$.
\end{itemize}
\end{theorem}
The measure $\RV$ (resp. $\RVa$) is called \emph{the Riesz measure on the cone $\CV$} (resp. $\CVa$). {The sets} $\Xi$ and $\mathfrak{X}$ are called \emph{the Gindikin-Wallach sets}.

 Riesz measures were described explicitly in \cite{Is00}. {The measure} $\RV$ (resp. $\RVa$) is singular unless $\us\in \Xi(1,\ldots,1)$ (resp. $\us\in \mathfrak{X}(1,\ldots,1)$). If $\us\in\Xi(1,\ldots,1)$ (resp. $\us\in\mathfrak{X}(1,\ldots,1)$), then the Riesz measure is an absolutely continuous measure with respect to the Lebesgue measure. In such case, the support of $\RV$ (resp. $\RVa$) equals $\overline{\CV}$ (resp. $\overline{\CVa}$).

We are interested in the description of natural exponential families generated by $\RV$ and $\RVa$. 
Members of $F(\RV)$ and $F(\RVa)$ are called \emph{Wishart distributions on $\CV$} and $\CVa$, respectively. In order to define NEFs generated by $\RV$ and $\RVa$ we have to ensure that  $\RV\in\mathcal{M}(\ZV)$ and $\RVa\in\mathcal{M}(\ZVa)$ at least for some $\us$. We have the following

\begin{theorem}[\mbox{\cite[Theorem 3.4]{IK15}}]\label{hyper}
\begin{itemize}
\item[\rm (i)]
Let $\us\in\Xi$.
The support of $\RV$ is not {concentrated} on any affine hyperplane in $\ZV$ if and only if $s_k>0$ for all $k=1,\ldots,r$.
\item[\rm (ii)]
Let $\us\in\mathfrak{X}$.
The support of $\RVa$ is not {concentrated} on any affine hyperplane in $\ZVa$ if and only if $s_k>0$ for all $k=1,\ldots,r$.
\end{itemize}
\end{theorem}

\subsection{Group equivariance of  the Wishart exponential families}
We say that a measure $\mu$ on $\E$ is \emph{relatively invariant}
 under a subgroup $G$ of $\GL(\E)$, if for all $g\in G$ there exists a constant $c_g>0$ for which $\mu(gA)=c_g\mu(A)$ for any measurable $A\subset\E$. This condition is equivalent to
$$L_\mu(g^\ast \theta)=c_g^{-1}L_\mu(\theta),\qquad \theta\in\Theta(\mu),$$
where $g^\ast$ is the {adjoint of} $g$.

 Formulas \eqref{BigDel} and \eqref{SmallDel} imply that
 Riesz measure $\RV$ is invariant under the group $\rho(\HV)$, 
while the dual Riesz measure $\RVa$ is invariant under $\rho^\ast(\HV)$. 
It  follows  that the Wishart exponential family  ${F(\RV)}$ is 
invariant under $\rho(\HV)$  and, analogously,  ${F(\RVa)}$ is 
invariant under $\rho^\ast(\HV)$.

\section{{The inverse of the  mean map and the Lauritzen formula on  $\CVa$}} \label{SectionInvMean}

Let $\us\in\mathfrak{X} \cap \R^r_{>0}$.
Then we have $\RVa\in\mathcal{M}(\ZVa)$ by Theorem \ref{hyper}.
Denote by $\psi_{\us}:=\psi_{\RVa}$ the inverse of the mean map from $M_{F(\RVa)}$ to $\Theta(\RVa)$.
In this section, we give an explicit formula for $\psi_{\us}(m),\,\,\,m \in M_{F(\RVa)}$. 
Thanks to Theorem \ref{riesz}, we have $\CV\subset \Theta(\RVa)$ and $M_{F(\RVa)}\subset\CVa$.
Applying \cite[Proposition IV.3]{Is14}, we can show that $\CV= \Theta(\RVa)$ and $M_{F(\RVa)}=\CVa$.
Indeed, it suffices to check that, for any sequence $\{y_k\}_{k \in \mathbb{N}}$ in $\CV$ converging to a point in $\partial \CV$,
 we have $\lim_{k \to \infty}\Delta_{-\us}(y_k) = + \infty$ because $\underline{s} \in \R^r_{>0}$.

\begin{proposition}\label{ProPsi}
	The inverse of the mean map on $\CVa$ is expressed by
\begin{align} \label{psi}
\psi_{\us}(m)=-(\log\delta_{- \us})'(m).
\end{align}
\end{proposition}

\begin{proof}
By Lemma \ref{lem:Jmu} we obtain $\psi_{\us}=-(\log J_{\RVa})'$. 
For $T \in \HV$, we have
\begin{align*}
J_{\RVa}(\rho^\ast(T)I_N) 
= \sup_{\theta \in \CV} \frac{e^{- \scalar{\rho^\ast(T)I_N,\theta}}}{L_{\RVa}(\theta)}
= \sup_{\theta \in \CV} \chi_{-\us}(T)\frac{e^{- \scalar{I_N,\rho(T)\theta}}}{L_{\RVa}(\rho(T)\theta)},
\end{align*}
where the last equality follows from \eqref{BigDel}.
Since $\rho(T) \CV = \CV$, 
we get
$$
J_{\RVa}(\rho^\ast(T)I_N) 
= \chi_{-\us}(T) \cdot \sup_{\theta \in \CV} \frac{e^{- \scalar{I_N,\theta}}}{L_{\RVa}(\theta)}
= \delta_{-\us}(\rho^\ast(T)I_N) J_{\RVa}(I_N).
$$
We see that 
the function  $\delta_{- \us}$ equals $J_{\RVa}$
up to a constant multiple.
Therefore $-(\log\delta_{- \us})'$ coincides with $-(\log J_{\RVa})'$.
\end{proof}
\begin{remark}
It is shown in \cite[Proposition 3.16]{No03} 
   that $- (\log \Delta_{-\us})'$ gives a diffeomorphism from the homogeneous cone $\CV$ onto $\CVa$
   for any $\us \in \R_{>0}^r$, 
   and that $- (\log \delta_{-\us})'$ gives the inverse map of $- (\log \Delta_{-\us})'$.
Proposition \ref{ProPsi} follows from this fact 
 because the mean map $m_{\RVa}$ equals  $-(\log L_{\RVa})' = - (\log \Delta_{-\us})'$ for $\us \in \Xi\cap\R_{>0}^r$
 by (\ref{eqn:LT-Rs}).
Nevertheless, we have given a short simple proof of Proposition \ref{ProPsi} for completeness. 
\end{remark}

Let us evaluate $\psi_{\us}(m) \in \CV$ for $m \in \CVa$.
In general, for a positive integer $M$, 
we regard the set $\mathrm{Sym}(M, \R)$ of $M \times M$ symmetric matrices
as a Euclidean vector space 
with the trace inner product $\tr XY \,\,\,(X, Y \in \mathrm{Sym}(M, \R))$.
Then we consider
the linear map $\phi_i^\ast : \mathrm{Sym}(m_i, \R) \to \ZV,\,\,\,i=1, \ldots, r,$
adjoint to $\phi_i : \ZV \to \mathrm{Sym}(m_i, \R)$ defined in such a way that
$$
\scalar{\xi, \phi_i^\ast(X)}= \tr \phi_i(\xi) X, \qquad {X \in \mathrm{Sym}(m_i, \R),\,\,\xi \in \ZV}.
$$
The linear map $\widecheck{\phi}_i^\ast : \mathrm{Sym}(m_i-1, \R) \to \ZV,\,\,\,i=1, \ldots, r,$
adjoint to $\widecheck{\phi}_i : \ZV \to \mathrm{Sym}(m_i-1, \R)$ is defined similarly.
{
\begin{theorem}\label{InvMean}
		The inverse of the mean map on $\CVa$ is given by the formula
\begin{equation} \label{eqn:psis_formula}
\psi_{\us}(m) = s_r \phi_r^\ast(\phi_r(m)^{-1}) + \sum_{i=1}^{r-1} s_i  \left(\phi_i^\ast( \phi_i(m)^{-1}) - \widecheck{\phi}_i^\ast(\widecheck{\phi}_i(m)^{-1})\right).
\end{equation}
\end{theorem}
}
\begin{proof}
For any $\alpha \in \ZV$,
	we see from \eqref{psi} and Proposition \ref{proPhi} (iii) that 
	\begin{multline*}
	\scalar{\alpha, \psi_{\us}(m)} 
	= - D_{\alpha}\log \delta_{-\us}(m)
	=  s_r D_{\alpha} \log \phi_r(m) + \sum_{i=1}^{r-1} s_i D_{\alpha} \left( \log \det \phi_i(m) - \log \det \widecheck{\phi}_i(m) \right),
	\end{multline*}
	where $D_\alpha$ denotes the directional derivative in the direction of $\alpha$.
	By the well-known formula for the derivative of log-determinant, 
	we get
	\begin{equation} \label{eqn:phi_inverse}
	\scalar{\alpha, \psi_{\us}(m)} =
	s_r \frac{\phi_r(\alpha)}{\phi_r(m)} + \sum_{i=1}^{r-1} s_i \left( \tr \phi_i(\alpha) \phi_i(m)^{-1} - \tr \widecheck{\phi}_i(\alpha) \widecheck{\phi}_i(m)^{-1} \right).
	\end{equation}
	Using the adjoint  maps, we rewrite \eqref{eqn:phi_inverse} as
	\begin{multline*}
	\scalar{\alpha, \psi_{\us}(m)} =
	s_r \scalar{\alpha, \phi_r^\ast(\phi_r(m)^{-1})} + 
	\sum_{i=1}^{r-1} s_i \left( \scalar{\alpha, \phi_i^\ast( \phi_i(m)^{-1})} 
	- \scalar{\alpha, \widecheck{\phi}_i^\ast(\widecheck{\phi}_i(m)^{-1})} \right),
	\end{multline*}
	so that we obtain formula \eqref{eqn:psis_formula}.

	\end{proof}

	Let $\un := (n_1, \ldots, n_r)$. 
	Noting that $\hat{m}^{-1} \in \CV$,
	we have
	$$\det \hat{m}^{-1} = \Delta_{\un}(\hat{m}^{-1}) = \delta_{-\un}(m)$$
	by (\ref{eqn:Delta_s}) and (\ref{eqn:Delta_delta2}).
	Thus, for $\alpha \in \ZV$ we observe
	$$
	\scalar{\alpha, \psi_{\un}(m)} = - D_{\alpha}\log \delta_{-\un}(m) = - D_{\alpha}\log \det \hat{m}^{-1} = \tr \alpha \hat{m}^{-1} = \scalar{\alpha, \hat{m}^{-1}}.
	$$
	Therefore, by \eqref{eqn:psis_formula} we get
	\begin{corollary}
		{The inverse of the bijection $y{\mapsto} \pi(y^{-1}), \CV\rightarrow \CVa$ is given explicitly by}
	\begin{equation} \label{eqn:gen_Lauritzen}
	m\mapsto\hat{m}^{-1} = \psi_{\un}(m) = n_r \phi_r^\ast(\phi_r(m)^{-1}) + \sum_{i=1}^{r-1} n_i  \Bigl(\phi_i^\ast( \phi_i(m)^{-1}) - \widecheck{\phi}_i^\ast(\widecheck{\phi}_i(m)^{-1})\Bigr).
	\end{equation}
	\end{corollary}
{	If $n_1 = n_2 = \ldots = n_r = 1$, then \eqref{eqn:gen_Lauritzen} yields the Lauritzen formula {\cite[(5.21)]{Lbook}} for homogeneous graphical cones
	(cf. Example \ref{Vin}).} 
{Formula \eqref{eqn:gen_Lauritzen} generalizes the Lauritzen formula to all homogeneous cones.}

\section{Variance function of Wishart exponential families on $\CVa$}
\label{VarQ}
As in the previous section, let $\us \in \mathfrak{X} \cap \R_{>0}^r$.
\begin{lemma} \label{LEMequiVar}
	The variance functions of the Wishart exponential families
	satisfy   
	\begin{align}\label{equiVar}
	\V_{F(\RVa)}(\rho^\ast(T)I_N)=\rho^\ast(T)\V_{F(\RVa)}(I_N)\rho(T),\qquad T\in\HV,
	\end{align}
	\begin{align}\label{equiVar2}
	\V_{F(\RV)}(\rho(T)I_N)=\rho(T)\V_{F(\RV)}(I_N)\rho^\ast(T),\qquad T\in\HV.
	\end{align}
\end{lemma}

\begin{proof}
	The identity \eqref{equiVar} for the variance function
	follows  from the 
	invariance of $F(\RVa)$ under $\rho(\HV)$   (see for example formula (2.2) in \cite{IK15}).
	The invariance property of $\RV$ results in identity \eqref{equiVar2}.
\end{proof}

In \cite[Theorem 7]{IK15} it was shown that property \eqref{equiVar2}
actually characterizes measure $\RV$. 
The same is true for \eqref{equiVar} and $\RVa$.

Recall that $N=n_1+\ldots+n_r=N_r$. 
For $z\in\mathrm{Sym}(N_k,\R)$, we define the matrix $z_0\in\mathrm{Sym}(N,\R)$ completed with zeros, that is, $(z_0)_{\{1:k\}}=z$ and \mbox{$(z_0)_{ij}=0$} if $\max\{i,j\}>N_k$.
Set $J_k:=\left(I_{N_k}\right)_0\in\ZV$ and $J_k^\ast:=I_N-J_k$.
\begin{proposition}\label{proHat}
If $y=T^\top T\in\mathrm{Sym}_+(N,\R)$ with $T\in\HV$, then
$$T^\top J_k^\ast\,\,  T =y-\left[ \left(y^{-1}\right)_{\{1:k\}}\right]^{-1}_0.$$
\end{proposition}
\begin{proof}
We have to show that $T^\top J_k\,\,T=\left[ \left(y^{-1}\right)_{\{1:k\}}\right]^{-1}_0$. Observe first that
$$T^\top J_k\,\,T=\left(T_{\{1:k\}}^\top T_{\{1:k\}}\right)_0.$$
Set $S=T^{-1}\in\HV$. Then $y^{-1}=SS^\top$. Since $(SS^\top)_{\{1:k\}}=S_{\{1:k\}}S_{\{1:k\}}^\top$,
 we have
 $$\left[(y^{-1})_{\{1:k\}}\right]^{-1}=(S_{\{1:k\}}^\top)^{-1} (S_{\{1:k\}})^{-1}.$$
Thanks to $(S_{\{1:k\}})^{-1}=(S^{-1})_{\{1:k\}}=T_{\{1:k\}}$,
we get the assertion.
\end{proof}

Now we are ready to state and prove our main theorem.
\begin{theorem}\label{main}
Let $\us\in\mathfrak{X} \cap \R^r_{>0}$. Then, the variance function of $F(\RVa)$ is given by ($m\in\CVa$)
$$
\V_{F(\RVa)}(m)=\pi\circ\left\{ \frac{n_1}{s_1}\rho(\hat{m})+\sum_{i=2}^r\left(\frac{n_i}{s_i}-\frac{n_{i-1}}{s_{i-1}}\right)\rho\left(\hat{m}-\left[ \left(\hat{m}^{-1}\right)_{\{1:i-1\}}\right]^{-1}_0\right)\right\}.
$$
\end{theorem}
\begin{proof}

{Similarly {to} the proof of Theorem \ref{InvMean}, using
		formula \eqref{psi} from Proposition \ref{ProPsi},
		 we obtain}
\begin{multline*}
\scalar{\alpha,\psi_{\us}'(m)\beta}=-D_{\alpha,\beta}^2 \log\delta_{-\us}(m)\\
=s_r D^2_{\alpha,\beta}\log \phi_r(m)+\sum_{i=1}^{r-1} s_i\left( D^2_{\alpha,\beta}\log\det \phi_i(m)-D^2_{\alpha,\beta}\log\det \widecheck{\phi}_i(m)\right)\\
 =-s_r \frac{\phi_r(\alpha)\phi_r(\beta)}{\phi_r(m)^{2}}- \sum_{i=1}^{r-1} s_i \left(\tr \phi_i(\alpha)\phi_i(m)^{-1}\phi_i(\beta)\phi_i(m)^{-1} \phantom{\widecheck{\phi}_i(m)^{-1}}\right. \\
 \left.-\tr\widecheck{\phi}_i(\alpha)\widecheck{\phi}_i(m)^{-1}\widecheck{\phi}_i(\beta)\widecheck{\phi}_i(m)^{-1}  \right),
\end{multline*}
{where $D^2_{\alpha,\beta} =D_{\alpha} D_{\beta}$.}
Setting $m=I_N$, we obtain by Proposition \ref{proPhi} \rm{(iv)}
$$\scalar{\alpha,\psi_{\us}'(I_N)\beta}=-s_r\alpha_{rr}\beta_{rr}-\sum_{i=1}^{r-1} \frac{s_i}{n_i}\left(n_i\alpha_{ii}\beta_{ii}+2\sum_{l=i+1}^r n_l(\alpha_{li}|\beta_{li})\right).$$
Recall that $J_k^\ast=\begin{pmatrix}0_{N_k} & \\ & I_{n_{k+1}+\ldots+n_r}\end{pmatrix}\in\ZV$ and set $\PP_k:=\rho(J_k^\ast)$, $k=1,\ldots,r-1$, $\PP_0=\mathrm{Id}_{\ZV}$. $\PP_i$ is the orthogonal projection onto $\bigoplus_{i<k,l\leq r} \VV_{lk}$.
Then, through {a} direct computation, one can show that for \mbox{$i=1,\ldots,r-1$},
$$n_i\alpha_{ii}\beta_{ii}+2\sum_{l=i+1}^r n_l(\alpha_{li}|\beta_{li}) = \scalar{\alpha,(\PP_{i-1}-\PP_{i}) \beta},$$
and
$$n_r\alpha_{rr}\beta_{rr}=\scalar{\alpha,\PP_{r-1}\beta}.$$
These imply that ($\PP_r:=0_{\ZV}$)
$$\psi_{\us}'(I_N)=-\sum_{i=1}^r \frac{s_i}{n_i} (\PP_{i-1}-\PP_{i}).$$
Since $(\PP_{i-1}-\PP_{i})(\PP_{j-1}-\PP_{j})=\delta_{ij}(\PP_{i-1}-\PP_{i})$, we have
\begin{align}\label{psiinv}
[\psi_{\us}'(I_N)]^{-1}=-\sum_{i=1}^r \frac{n_i}{s_i} (\PP_{i-1}-\PP_{i}).
\end{align}
Indeed,
$$\psi_{\us}'(I_N)\sum_{j=1}^r \left(-\frac{n_j}{s_j} (\PP_{j-1}-\PP_{j})\right)=\sum_{i=1}^r (\PP_{i-1}-\PP_{i})=\PP_0=\mathrm{Id}_{\ZV}.$$
Thus, \eqref{psiinv} gives us 
$$\V_{{F(\RVa)}}(I_N)=-[\psi_{\us}'(I_N)]^{-1}=\frac{n_1}{s_1}\mathrm{Id}_{\ZV}+\sum_{i=2}^r\left(\frac{n_i}{s_i}-\frac{n_{i-1}}{s_{i-1}}\right)\PP_{i-1}.$$
Finally, using \eqref{equiVar} we obtain for $m=\rho^\ast(T)I_N$,
\begin{align}\label{VformT}
\begin{split}
\V_{{F(\RVa)}}(m)&=\rho^\ast(T)\V_{{F(\RVa)}}(I_N)\rho(T)\\
&=\frac{n_1}{s_1}\rho^\ast(T)\rho(T)+\sum_{i=2}^r\left(\frac{n_i}{s_i}-\frac{n_{i-1}}{s_{i-1}}\right)\rho^\ast(T)\PP_{i-1}\rho(T).
\end{split}
\end{align}
Since $\rho^\ast(T)=\pi\circ \rho(T^\top)$ and $\hat{m}=T^\top T$, we have
$$\rho^\ast(T)\rho(T)=\pi\circ\rho(T^\top)\rho(T)=\pi\circ\rho(T^\top T)=\pi\circ\rho(\hat{m}),$$
and, by Proposition \ref{proHat}, for $i=2,\ldots,r$, 
$$\rho^\ast(T)\PP_{i-1}\rho(T)=\pi\circ\rho(T^\top J_{i-1}^\ast T)=\pi\circ\rho\left(\hat{m}-\left[ \left(\hat{m}^{-1}\right)_{\{1:i-1\}}\right]^{-1}_0\right).$$
\end{proof}
{\begin{remark}
		Note that the formula  for the inverse of the mean map
		given in Theorem \ref{InvMean}
	{ is not  necessary for the proof of Theorem \ref{main}.}
	{	Formula \eqref{psi} from Proposition \ref{ProPsi} 
		is sufficient.}
		\end{remark}}	
\begin{example}\label{Vin}
Let us apply Theorem \ref{main} to the Wishart exponential families on the Vinberg cone. 
Let
$$\ZV:=\left\{ \begin{pmatrix}x_{11} & 0 & x_{31} \\ 0 & x_{22} & x_{32} \\ x_{31} & x_{32} & x_{33} \end{pmatrix}\colon x_{11},x_{22},x_{33},x_{31},x_{32}\in\R\right\}.$$
Conditions {\rm (V1)--(V3)} are satisfied and we have $n_1=n_2=n_3=1$, $N=r=3$.
Then,
$$\CV=\ZV\cap \mathrm{Sym}_+(3,\R)=\left\{ x\in\ZV\colon x_{11}>0,\,x_{22}>0,\,\det x>0\in\R\right\}$$
and its dual cone is given by 
$$\CVa=\left\{ \xi\in \ZV \colon \xi_{33}>0,\,\xi_{11}\xi_{33}>\xi_{31}^2,\,\xi_{22}\xi_{33}>\xi_{32}^2\right\}.$$ 
{The cone $\CVa$}
 is called the Vinberg cone, while $\CV$ is called the dual Vinberg cone.
The cones $\CVa$ and $\CV$  are the lowest dimensional non-symmetric homogeneous cones.

{
For $\xi \in \ZV$, we have
\begin{gather*}
\phi_1(\xi) = \begin{pmatrix} \xi_{11} & \xi_{31} \\ \xi_{31} & \xi_{33} \end{pmatrix} =: \xi_{\{1,3\}},\qquad
\phi_2(\xi) = \begin{pmatrix} \xi_{22} & \xi_{32} \\ \xi_{32} & \xi_{33} \end{pmatrix} =: \xi_{\{2,3\}},\\
\phi_3(\xi) = \widecheck{\phi}_1(\xi) = \widecheck{\phi}_2(\xi) = \xi_{33}, 
\end{gather*}
 so that
\begin{gather*}
\phi_1^\ast \begin{pmatrix} a & b \\ b & c \end{pmatrix}
= \begin{pmatrix} a & 0 & b \\ 0 & 0 & 0 \\ b & 0 & c\end{pmatrix}, \qquad
\phi_2^\ast \begin{pmatrix} a & b \\ b & c \end{pmatrix}
= \begin{pmatrix} 0 & 0 & 0 \\ 0 & a & b \\ 0 & b & c\end{pmatrix}, \\
\phi_3^\ast(a) 
 = \widecheck{\phi}_1^\ast(a) = \widecheck{\phi}_2^\ast(a)
 = \begin{pmatrix} 0 & 0 & 0 \\ 0 & 0 & 0 \\ 0 & 0 & a\end{pmatrix} 
\end{gather*}
 for $a,b, c \in \R$.
Therefore, for $m \in \CVa$ and $\us = (s_1, s_2, s_3)$ we obtain by \eqref{eqn:psis_formula}
\begin{align*}
 \psi_{\us}(m) 
= s_1 \begin{pmatrix} \frac{m_{33}}{|m_{\{1,3\}}|} & 0 & -\frac{m_{31}}{|m_{\{1,3\}}|}\\ 0 & 0 & 0 \\
      -\frac{m_{31}}{|m_{\{1,3\}}|} & 0 &  \frac{m_{11}}{|m_{\{1,3\}}|} \end{pmatrix}
 &+s_2 \begin{pmatrix} 0 & 0 & 0 \\ 0 & \frac{m_{33}}{|m_{\{2,3\}}|} & -\frac{m_{32}}{|m_{\{2,3\}}|}\\
      0 & -\frac{m_{32}}{|m_{\{2,3\}}|} & \frac{m_{22}}{|m_{\{2,3\}}|} \end{pmatrix} \\
& + (s_3 - s_1 - s_2) \begin{pmatrix} 0 & 0 & 0 \\ 0 & 0 & 0 \\ 0 & 0 & \frac{1}{m_{33}} \end{pmatrix}.
\end{align*}
In particular, \eqref{eqn:gen_Lauritzen} tells us that
\begin{align} \label{eqn:hatm_inverse}
 \hat{m}^{-1} 
= \begin{pmatrix} \frac{m_{33}}{|m_{\{1,3\}}|} & 0 & -\frac{m_{31}}{|m_{\{1,3\}}|}\\ 0 & 0 & 0 \\
      -\frac{m_{31}}{|m_{\{1,3\}}|} & 0 &  \frac{m_{33}}{|m_{\{1,3\}}|} \end{pmatrix}
 + \begin{pmatrix} 0 & 0 & 0 \\ 0 & \frac{m_{33}}{|m_{\{2,3\}}|} & -\frac{m_{32}}{|m_{\{2,3\}}|}\\
      0 & -\frac{m_{32}}{|m_{\{2,3\}}|} & \frac{m_{22}}{|m_{\{2,3\}}|} \end{pmatrix}
 - \begin{pmatrix} 0 & 0 & 0 \\ 0 & 0 & 0 \\ 0 & 0 & \frac{1}{m_{33}} \end{pmatrix}
\end{align}
where $|m_{\{1,3\}}|=m_{11}m_{33}-m_{31}^2$ and $|m_{\{2,3\}}|=m_{22}m_{33}-m_{32}^2$.
This is exactly the Lauritzen formula.
Moreover, we have}
$$
\hat m=
\begin{pmatrix}
m_{11} & \frac{m_{31}m_{32}}{m_{33}} & m_{31} \\
  \frac{m_{31}m_{32}}{m_{33}} & m_{22} & m_{32} \\ m_{31} & m_{32} & m_{33}
\end{pmatrix}
$$
and it is easy to see that the projection $\pi\colon \mathrm{Sym}(3,\R)\to\ZV$ sets $0$ in $(1,2)$ and $(2,1)$ entries, leaving all other entries unchanged.

We shall use Theorem \ref{main} in order to give {$\V_{F(\RVa)}$ explicitly}. 
We denote $M_i= \frac{|m_{\{i,3\}}|}{m_{33}}E_{ii}$ {for $i=1,2$},
 where $E_{ii}$ is the diagonal $3\times 3$ matrix with $1$ in $(i,i)$ entry and all else $0$.
{
Then we have by (\ref{eqn:hatm_inverse})
$$
 \left[ \left(\hat{m}^{-1}\right)_{\{1:1\}}\right]^{-1}_0 = M_1, \qquad
 \left[ \left(\hat{m}^{-1}\right)_{\{1:2\}}\right]^{-1}_0 = M_1 + M_2.
$$
}
Theorem \ref{main} gives for $\us\in\mathfrak{X}\cap \R^3_{>0}$,
\begin{equation}\label{form51}
\begin{split}
 \V_{F(\RVa)}(m) = \pi\circ\left\{ \frac{1}{s_1}\rho(\hat m) \right. +\left(\frac{1}{s_2}-\frac{1}{s_1}   \right)\rho(\hat m-M_1)
\left.+\left(\frac{1}{s_3}-\frac{1}{s_2}   \right)\rho(\hat m-M_1-M_2)\right\}.
\end{split}
\end{equation}
Elementary properties of the quadratic operator $\rho$ and of its bilinear extension
$\rho(a,b)x=\frac12(axb^{\top}+bxa^{\top})$, and the fact that $\rho(M_1, M_2)=0$ on $\ZV$,  imply the
following  formula,  proven by different methods in \cite{GIMamane}:
 \begin{align}\label{form2}\begin{split}
 \V_{F(\RVa)}(m)=\pi\circ\left\{\left( \frac{1}{s_1}+\frac{1}{s_2}-\frac{1}{s_3} \right) \rho(\hat m) 
 \right.+\left(\frac{1}{s_3}-\frac{1}{s_1}   \right)\rho(\hat m-M_1) 
\left.+\left(\frac{1}{s_3}-\frac{1}{s_2}   \right)\rho(\hat m-M_2)\right\}.
\end{split}\end{align}
Observe that formulas \eqref{form51} and \eqref{form2} imply analogous formulas for the homogeneous cone $Q$, dual to the cone $P$ in the vector space 
$$Z:=\left\{ \begin{pmatrix}x_{11}  & x_{21}&0 \\  x_{21} & x_{22} & x_{32} \\ 0& x_{32} & x_{33} \end{pmatrix}\colon x_{11},x_{21},x_{22},x_{32}, x_{33}\in\R\right\},$$
i.e. $P=Z\cap {\rm Sym}_+(3,\R)$. Note that $Z\cap {\rm Sym}_+(3,\R)$ is not a matrix realization
of the cone $P$, so Theorem \ref{main} does not apply {directly} to Wishart families on $Q$.  
Instead, we use the permutation  $(1,3,2)\mapsto (1,2,3)$ and $\hat m$ with $\hat m_{13}=\frac{{m_{21}}m_{32}}{m_{22}}$. 
For example, formula \eqref{form2} gives the following formula proven in \cite{GIMamane}
 \begin{align*}
\V_Q(m)=
\pi\circ\left\{  
\left( \frac{1}{s_1}+\frac{1}{s_3}-\frac{1}{s_2} \right) \rho(\hat m) 
\right. +\left(\frac{1}{s_2}-\frac{1}{s_1}   \right)\rho(\hat m-{M'_1})
\left.+\left(\frac{1}{s_2}-\frac{1}{s_3}   \right)\rho(\hat m-{M'_3})\right\} ,
\end{align*}
{where $M'_1 := \frac{x_{11} x_{22} - x_{21}^2}{x_{22}} E_{11}$ and $M'_3 := \frac{x_{22} x_{33} - x_{32}^2}{x_{22}} E_{33}$.}
Note that for $\us=(p,p,p)$ with $p>\frac12$, the variance function of $F(\RVa)$ on the Vinberg cone is
\begin{equation}\label{pconstant}
\V_p(m)=\frac1p\pi\circ\rho(\hat m).
\end{equation}
\end{example}

\begin{remark}
 A {false} formula for the variance function of a Wishart family on a homogeneous cone 
is announced in  the unpublished article \cite{BH09},  Theorem 4.2,
{see Appendix for more details.} 
\end{remark}
\section{Variance function of Wishart exponential families on \texorpdfstring{$\mathcal{P}_{\mathcal{V}}$}{PV} }\label{VarP}
We are going to find the variance function of the NEF generated by  $\RV$ on the cone $\CV$.
Using the similar approach (see \eqref{VformT}) as in the proof of Theorem \ref{main},
 one can show the following Proposition.

\begin{proposition} Let $\us\in\Xi\cap\R_{>0}^r$. For any $T\in\HV$ one has
$$\V_{F(\RV)}(\rho(T)I_N)=\frac{n_r}{s_r}\rho(T)\rho^\ast(T)+\sum_{k=1}^{r-1} \left(\frac{n_k}{s_{k}}-\frac{n_{k+1}}{s_{k+1}}\right) \rho(T)\rho(J_k)\rho^\ast(T).$$
\end{proposition}

\begin{proof}
We have $\Theta(\RV)=\CVa$ and $M_{F(\RV)}=\CV$. By definition,
$m_{\RV}(\theta)=-(\log\delta_{-\us})'(\theta)$, $\theta\in\CVa$. 
Use Lemma \ref{lem:Jmu} to show that $\psi_{\RV}(m)=-(\log\Delta_{-\us})'(m)$, $m\in\CV$ and proceed analogously as in the proof of Theorem \ref{main}.
\end{proof}

Now we will {also use} another approach to this problem. We will use the duality of the cones $\CV$ and $\CVa$ and a matrix realization of $\CVa$. The objective is to boil down to the results of the preceding Section and apply  the formula
for the variance function from Theorem \ref{main}.

Dual cone $\CVa$ to a homogeneous cone is also homogeneous. Thus, due to \cite[Theorem D]{Is06} it admits (under suitable linear isomorphism) a matrix realization.
There exists a family $\widetilde{\VV}=\{\widetilde{\VV}_{lk}\}_{1\leq k<l\leq \tilde{r}}$ satisfying \rm{(V1)-(V3)} and a linear isomorphism $l\colon \mathcal{Z}_{\widetilde{\VV}}\to\ZV$ such that $\CVa=l(\mathcal{P}_{\widetilde{\VV}})$. It can be shown that $\tilde{r}=r$.

Since $l(\mathcal{P}_{\widetilde{\VV}})=\CVa$ and $l$ is an isomorphism, for any $T\in \HV$ there exists a unique $\widetilde{T}\in \mathrm{H}_{\widetilde \VV}$ such that
$$
l({\rho(\widetilde{T})} I_{\tilde{N}}^{\widetilde{\VV}})
 ={\rho^\ast(T)} I_N^{\VV}.
$$
The linear isomorphism $l$ can be taken in such a way that if $T$ has diagonal $(t_{11},\ldots,t_{rr})$ then $\widetilde{T}$ has diagonal $(t_{rr},\ldots,t_{11})$ (see the choice of the permutation $w$ in Proposition \ref{pro:good_permutation}). In such case we have $\chi^{\VV}_{\us}(\widetilde{T})=\chi^{\widetilde{\VV}}_{\us^\ast}(T)$, where $\us^\ast=(s_r,\ldots,s_1)$. This implies the following Proposition
\begin{proposition}\label{pro:mx_rn_Qn}
There exists a family $\widetilde{\VV}=\{\widetilde{\VV}_{lk}\}_{1\leq k<l\leq r}$ satisfying {\rm(V1)-(V3)} and a linear isomorphism $l\colon \mathcal{Z}_{\widetilde{\VV}}\to\ZV$ such that $\CVa=l(\mathcal{P}_{\widetilde{\VV}})$
and \begin{align}\label{strange}
\Delta^{\widetilde{\VV}}_{\us^\ast}(x)=\delta^{\VV}_{\us}(l(x)),\qquad x\in \mathcal{P}_{\widetilde{\VV}}.
\end{align}
In this case, $\us\in\Xi_{\VV}$ if and only if $\us^\ast\in\mathfrak{X}_{\widetilde{\VV}}$. 
\end{proposition}

The adjoint map $l^\ast\colon\ZV\to\mathcal{Z}_{\widetilde{\VV}}^\ast\equiv \mathcal{Z}_{\widetilde{\VV}}$ is a linear isomorphism such that $l^\ast(\CV)=\mathcal{Q}_{\widetilde{\VV}}$
and
$$\Delta^{\VV}_{\us}(x)=\delta^{\widetilde{\VV}}_{\us^\ast}(l^\ast(x)),\qquad x\in \CV.$$
Consider the Riesz measure $\RV$ on $\CV$. Then for any $\xi\in\CVa$
$$L_{\RV}(\xi)=\delta^{\VV}_{-\us}(\xi)=\Delta^{\widetilde{\VV}}_{-\us^\ast}((l^\ast)^{-1}(\xi))=L_{\mathcal{R}_{\us^\ast}^\ast}((l^\ast)^{-1}(\xi)),$$
where $L_{\mathcal{R}_{\us^\ast}^\ast}$ is the Laplace transform of the Riesz measure on $\mathcal{P}_{\widetilde{V}}$. 
We have proven the following
\begin{theorem}\label{mainP}
Let $\us\in\Xi\cap\R_{>0}^r$. Then the variance function of $F(\RV)$ is given by 
$$\V_{F(\RV)}({\theta})=(l^\ast)^{-1}\circ \V_{F(\mathcal{R}_{\us^\ast}^\ast)}(l^\ast({\theta}))\circ l^{-1},\qquad \theta\in M_{F(\RV)}=\CV.$$
\end{theorem}
Here $\V_{F(\mathcal{R}_{\us^\ast}^\ast)}$ is the variance function of the Riesz measure defined on $\mathcal{P}_{\widetilde{\VV}}$, which can be written using Theorem \ref{main}.
The drawback of this result is that the map $l$ and so $l^\ast$ is generically not explicit. 
In the second part of this Section we propose a practical construction
	of a  matrix realization of the cone $\CVa$. Consequently, the maps 
 $l$ and  $l^\ast$ will be available and Theorem \ref{mainP} will become useful in statistical practice.

 \subsection{Matrix realization of the cone $\CVa$.}
One general way to get a matrix realization of the cone $\CVa$ as $\mathcal{P}_{\widetilde{\VV}}$ is as follows.
Recall that $m_i=\dim W_i$, which was defined in Section \ref{QUADR}. Noting that $\sum_{i=1}^r m_i$ equals the dimension $d$ of the cone $\CVa$,
we define a linear map 
$$ \Phi : \mathcal{Z}_{\VV} \owns \xi \mapsto
 \begin{pmatrix}
\phi_1(\xi)&&&\\&\phi_2(\xi)&&\\&&\ddots&\\&&&\phi_r(\xi) \end{pmatrix} \in \mathrm{Sym}(d,\R),
$$
and put
$$
 \mathcal{A}_{\VV} := \set{\Phi(I_N)^{-1/2} \Phi(\xi) \Phi(I_N)^{-1/2}}{\xi \in \mathcal{Z}_{\VV}}.
$$
\begin{proposition} \label{pro:dual_matrix_realization}
By an appropriate permutation of rows and columns,
 the subspace $\mathcal{A}_{\VV}$ of $\mathrm{Sym}(d,\R)$ gives a matrix realization of $\CVa$, 
thus we have $l^{-1}=\rho(w\, \Phi(I_d)^{-1/2})\circ\Phi$ for some permutation matrix $w \in \GL(d,\R)$.
\end{proposition}
\begin{proof}
We note that $\Phi(\xi)$ is positive definite if and only if $\xi \in \CVa$.
Indeed, the 'only if' part follows from the fact that $\xi \in \CVa$ is characterized by the positivity of $\det \phi_i(\xi)$
 for all $i=1, \ldots, r$ (see \cite[Proposition 3.1 (iv)]{GI14}).
To show Proposition \ref{pro:dual_matrix_realization},
 we shall introduce an algebra structure on the space $\mathrm{Sym}(N,\R)$ of symmetric matrices,
 and apply \cite[Theorem 2]{Is15}.
For $X \in \mathrm{Sym}(N,\R)$,
 let $\underline{X}$ be {the} unique lower triangular matrix for which $X = \underline{X} + \underline{X}^\top$.
For $X, Y \in \mathrm{Sym}(N,\R)$, define
$$
 X \triangle Y := \underline{X} Y + Y \underline{X}^\top \in \mathrm{Sym}(N,\R).
$$ 
It is not difficult to see that the conditions (V1) -- (V3) are satisfied
 if and only if the space $\ZV \subset \mathrm{Sym}(N,\R)$ forms a subalgebra of $(\mathrm{Sym}(N,\R), \triangle)$,
 that is, $X \triangle Y \in \ZV$ for all $X, Y \in \ZV$.
On the other hand,  \cite[Theorem 2]{Is15} states that,
 if $\mathcal{A} \subset \mathrm{Sym}(N, \R)$ is a subalgebra of $(\mathrm{Sym}(N,\R), \triangle)$ containing $I_N$,
 then there exists a permutation matrix $w \in \GL(N,\R)$ such that $\rho(w)\mathcal{A} = \set{\rho(w) X}{X \in \mathcal{A}}$
 is of the form $\mathcal{Z}_{\widetilde{\VV}}$ with some $\widetilde{\VV} = \{ \widetilde{\VV}_{lk} \}$. 
Therefore, for the proof of Proposition \ref{pro:dual_matrix_realization},
 it suffices to show that $\mathcal{A}^0_{\VV} := \set{ \rho(w_0) X}{X \in \mathcal{A}_{\VV}}$ is a subalgebra of 
  $(\mathrm{Sym}(d,\R), \triangle)$ containing $I_d$,
 where $w_0$  is the anti-diagonal matrix corresponding
 to the permutation $\begin{pmatrix} 1 & 2 & \cdots & d \\ d & d-1 & \cdots & 1 \end{pmatrix}$.
Clearly $I_d \in \mathcal{A}^0_{\VV}$ because $\Phi(I_{{N}})^{-1/2} \Phi(\xi) \Phi(I_{{N}})^{-1/2} = I_d$ with $\xi = I_N$.

For $T \in \HV$ and $x \in W_i$,
 then $T x \in W_i$ by {\rm(V1)}.
Thus, for each $T \in \HV$,
 there exists $\sigma_i(T) \in \GL(m_i, \R)$ such that $$\vc(Tx) = \sigma_i(T) \vc(x)$$ for all $x \in W_i$.
Here $\sigma_i(T)$ is a lower triangular matrix.
By Definition \ref{defi:phi_i}, we have for $\xi \in \ZV$
\begin{align*}
 \vc(x)^\top \phi_i(\rho^\ast(T)\xi) \vc(x)
 &= \scalar{\rho^\ast(T)\xi,q_i(x)} 
 = \scalar{\xi,\rho(T)q_i(x)} \\
 &= \scalar{\xi,q_i(Tx)}
 = \vc(Tx)^\top \phi_i(\xi)\vc(Tx)\\
 &= \vc(x)^\top \sigma_i(T)^\top \phi_i(\xi) \sigma_i(T)\vc(x),
\end{align*}
so that
$$
 \phi_i(\rho^\ast(T)\xi) = \sigma_i(T)^\top \phi_i(\xi) \sigma_i(T).
$$
For $T \in \HV$, we write
$$
\widetilde{T} := \begin{pmatrix} \sigma_1(T)^\top & & & \\ & \sigma_2(T)^\top & & \\ & & \ddots & \\ & & & \sigma_r(T)^\top \end{pmatrix}.
$$
Then $\widetilde{T}$ is an upper triangular matrix and we have 
\begin{equation}
 \Phi( \rho^\ast(T) \xi) = \widetilde{T} \Phi(\xi) \widetilde{T}^\top = \rho(\widetilde{T}) \Phi(\xi)
\qquad (T,\xi)\in \HV\times\ZV.
\end{equation}
Let $\mathcal{HA}_{\VV}$ be the set $\set{\Phi(I_N)^{-1/2}\tilde{T}\Phi(I_N)^{1/2} }{T \in \HV}$.
Then $\mathcal{HA}_{\VV}$ forms a Lie group and we have
$$
 \rho(S) X \in \mathcal{A}_{\VV}, \qquad (S,X) \in \mathcal{HA}_{\VV}\times \mathcal{A}_{\VV}.
$$
We observe that
 $$ \phi_i(I_N) = \begin{pmatrix} n_i & & & \\ & n_{i+1} I_{\dim \VV_{i+1, i}} & & \\ & & \ddots & \\& & &  n_r I_{\dim \VV_{r i}} \end{pmatrix}, $$
 so that $\Phi(I_N)$ is a diagonal matrix.
Thus, elements of $\mathcal{HA}_{\VV}$ are upper triangular matrices.
It follows that $\mathcal{HA}^0_{\VV} := \set{w_0 S w_0^{-1}}{S \in \mathcal{HA}_{\VV}}$ is a Lie group whose elements are lower triangular matrices
 and we have
$$
 \rho(S) X \in \mathcal{A}^0_{\VV}, \qquad (S,X)\in \mathcal{HA}^0_{\VV}\times\mathcal{A}^0_{\VV}.
$$ 
The $\rho(\mathcal{HA}^0_{\VV})$-orbit through $I_d$ coincides with the image of $\CVa \subset \ZV$
 by the linear map $\xi \mapsto \rho(w_0 \Phi(I_N)^{-1/2}) \Phi(\xi)$, so it is an open orbit.
Let $\mathrm{Lie}(\mathcal{HA}^0_{\VV})$ be the Lie algebra of the Lie group $\mathcal{HA}^0_{\VV}$.
Then $L \in \mathrm{Lie}(\mathcal{HA}^0_{\VV})$ is a lower triangular matrix, and the infinitesimal action is given by
 $\dot{\rho}(L) Y = L Y + Y L^\top \in \mathcal{A}^0_{\VV}$ for $Y \in  \mathcal{A}^0_{\VV}$.
In particular, we have a linear isomorphism
$$
 \mathrm{Lie}(\mathcal{HA}^0_{\VV}) \owns L \mapsto \dot{\rho}(L)I_d = L + L^\top \in \mathcal{A}^0_{\VV}, 
$$ 
 which implies that
$\mathrm{Lie}(\mathcal{HA}^0_{\VV}) = \set{\underline{X}}{X \in \mathcal{A}^0_{\VV}}$.
Therefore, for $X, Y \in \mathcal{A}^0_{\VV}$, we have
 $$
 X \triangle Y = \dot{\rho}(\underline{X}) Y \in \mathcal{A}^0_{\VV},
 $$ 
 so that $\mathcal{A}^0_{\VV}$ is a subalgebra of $(\mathrm{Sym}(d,\R), \triangle)$.
Hence Proposition \ref{pro:dual_matrix_realization} is verified.
\end{proof}

The permutation giving a matrix realization in Proposition \ref{pro:dual_matrix_realization} is not unique.
We shall present a practical method of finding such a permutation.
For $k=1, \ldots r$, we put $\nu_k := 1 + \sum_{i<k} \dim \VV_{ki}$.

\begin{proposition} \label{pro:good_permutation}
Let $w \in \GL(d,\R)$ be a permutation matrix such that
\begin{equation} \label{eqn:diag_assumption}
 \rho(w\, \Phi(I_N)^{-1/2})\Phi(\xi) 
 = \begin{pmatrix} \xi_{rr} I_{\nu_r} & & & \\ & \xi_{r-1,r-1} I_{\nu_{r-1}} & & \\ & & \ddots & \\ & & & \xi_{11} I_{\nu_1} \end{pmatrix} 
\end{equation}
holds for any diagonal $\xi \in \ZV$.
Then $\rho(w) \mathcal{A}_{\VV} = \mathcal{Z}_{\widetilde{\VV}}$
 with some vector spaces $\widetilde{\VV}_{lk} \subset \mathrm{Mat}(\nu_l, \nu_k,\R)\,\,\,(1 \le k < l \le r)$, which satisfy {\rm (V1)--(V3)}.
\end{proposition}

\begin{proof}
In the proof of Proposition \ref{pro:dual_matrix_realization}, we find a permutation matrix $w$ such that
 $\rho(w) \mathcal{A}_{\VV} = \mathcal{Z}_{\widetilde{\VV}}$ by applying \cite[Theorem 2]{Is15}.
And, in view of the proof of \cite[Theorem 2]{Is15},
 we see that this $w$ satisfies condition \eqref{eqn:diag_assumption}.
Let $w' \in \GL(d, \R)$ be another permutation matrix satisfying (\ref{eqn:diag_assumption}).
Then there exist permutation matrices $u_k \in \GL(\nu_k, \R)$ for $k=1, \ldots, r$ such that
$$
 w' = u w\quad \mbox{ with }\quad u := \begin{pmatrix} u_1 & & & \\ & u_2 & & \\ & & \ddots & \\ & & & u_r \end{pmatrix} \in \GL(d, \R).
$$
Then $\rho(u) \mathcal{Z}_{\widetilde{\VV}} = \mathcal{Z}_{\widetilde{\VV}'}$,
 where $\widetilde{\VV}'_{lk} := u_l \widetilde{\VV}_{lk} u_k^{\top}\,\,\,(1 \le k < l \le r)$.
On the other hand,
 noting that  $u_k^{\top} = u_k^{-1}$,  
 we see that the family of $\widetilde{\VV}'_{lk}$ satisfies {\rm(V1)-(V3)}.
In conclusion,
 we have
$$ \rho(w') \mathcal{A}_{\VV} =  \rho(u) \mathcal{Z}_{\widetilde{\VV}} =  \mathcal{Z}_{\widetilde{\VV}'}, $$
 which means that any permutation $w'$ satisfying (\ref{eqn:diag_assumption}) gives a matrix realization of $Q_{\VV}$.
\end{proof}
Let us note that condition \eqref{eqn:diag_assumption} is rather easy to check, since the matrix $\rho(\Phi(I_N)^{-1/2})\Phi(\xi)$ is diagonal if $\xi$ is diagonal.


We remark that some $\phi_i$'s can be omitted from $\Phi$: this is the ``optimization'' of matrix realizations of homogeneous cones,  recently developed 
by Nomura and Yamasaki \cite{NY15}, see Example \ref{DualVin}  below.

\begin{example}\label{DualVin}
Consider the same space $\ZV$ as in Example \ref{Vin}.
The variance function of a Wishart family on $\CVa$
was derived in Example \ref{Vin}. The objective of {Example \ref{DualVin}} is to give the variance function of a Wishart family on $\CV$. We start with constructing a convenient matrix realization of $\CVa$, using Propositions \ref{pro:mx_rn_Qn}, \ref{pro:dual_matrix_realization} and \ref{pro:good_permutation}.

For the dual Vinberg cone, the map $\Phi$ is {defined} as follows (see \cite{Is14}):
$$
\begin{pmatrix}
\xi_1& &\xi_4\\ &\xi_2&\xi_5\\ \xi_4&\xi_5&\xi_3
\end{pmatrix}
\rightarrow
\begin{pmatrix}
\xi_1&\xi_4&&&\\ \xi_4&\xi_3&&&\\&&\xi_2&\xi_5&\\&&\xi_5&\xi_3&\\&&&&\xi_3
\end{pmatrix}.
$$
Since $\Phi(I_3) = I_5$ in this case,
 $\mathcal{A}^0_{\VV}$ is the set of matrices of the form
$$
 \begin{pmatrix} \xi_3 & & & & \\ & \xi_3 & \xi_5 & & \\ & \xi_5 & \xi_2 & & \\ & & & \xi_3 & \xi_4 \\ & & & \xi_4 & \xi_1 \end{pmatrix}
$$
 and applying Proposition \ref{pro:good_permutation}, we get
 a matrix realization of $\CVa$ is given on the space $\mathcal{Z}_{\widetilde{\VV}}$ {of the matrices }of the form 
\begin{equation} \label{eqn:one_realization}
\begin{pmatrix}
\xi_3&&&&\\ &\xi_3&&\xi_5&\\ &&\xi_3&&\xi_4\\
&\xi_5&&\xi_2&\\
&&\xi_4&&\xi_1
\end{pmatrix}.
\end{equation}
Thanks to Proposition \ref{pro:good_permutation}, we have another matrix realization of $\CVa$ as {the space of positive definite matrices of the form}
$$
\begin{pmatrix} \xi_3 & & & & \xi_4 \\ & \xi_3 & & & \\ & & \xi_3 & \xi_5 & \\  & & \xi_5 & \xi_2 & \\ \xi_4 & & & & \xi_1 \end{pmatrix}.
$$
On the other hand, 
an
 optimal matrix realization is obtained by omitting the first diagonal $\xi_3$ in (\ref{eqn:one_realization}). In order to give the variance
 	function for Wishart families on the dual Vinberg cone $\CV$
 	we will use  this  optimal matrix realization.   The  isomorphism $l: \mathcal{Z}_{\widetilde{\VV}} \rightarrow \ZV$ is then  such that
$$
l^{-1}\colon 
\ZV\ni\begin{pmatrix}
\xi_1& &\xi_4\\ &\xi_2&\xi_5\\ \xi_4&\xi_5&\xi_3
\end{pmatrix}
\rightarrow
\begin{pmatrix}
\xi_3&&\xi_5&\\ &\xi_3&&\xi_4\\
\xi_5&&\xi_2&\\
&\xi_4&&\xi_1
\end{pmatrix}\in\mathcal{Z}_{\widetilde{\VV}}.
$$
It is easy to verify by direct computation that
$\delta^{\VV}_{\us}(l(x))=\Delta^{\widetilde{\VV}}_{\us^\ast}(x)$ for any $x\in\mathcal{P}_{\widetilde{\VV}}$, where $\us^\ast=(s_r,\ldots,s_1)$.
This relation is ensured by formula \eqref{strange}. The adjoint map $l^\ast: \ZV \rightarrow \mathcal{Z}_{\widetilde{\VV}} $ is, for $\theta\in \ZV$
	$$
l^\ast\begin{pmatrix}
\theta_1& &\theta_4\\ &\theta_2&\theta_5\\ \theta_4&\theta_5&\theta_3
\end{pmatrix}=
\begin{pmatrix}
\frac{\theta_3}2&&\theta_5&\\ &\frac{\theta_3}2&&\theta_4\\
\theta_5&&\theta_2&\\
&\theta_4&&\theta_1
\end{pmatrix},\ 
(l^\ast)^{-1}\begin{pmatrix}
x&&u&\\ &x&&v\\
u&&y&\\
&v&&z
\end{pmatrix}=
\begin{pmatrix}
z&0&v\\0&y&u\\v&u&2x
\end{pmatrix}.
$$
For $\mathcal{Z}_{\widetilde{\VV}}$ we have $\tilde{N}=4$, $\tilde{r}=3$, $\tilde{n}_1=2$ and $\tilde{n}_2=\tilde{n}_3=1$.
	Theorems \ref{mainP}
	and \ref{main} imply that for $\us\in\Xi$ and $\theta\in\CV$ 
$$\V_{F(\RV)}({\theta})=(l^\ast)^{-1}\circ \V_{F(\mathcal{R}_{\us^\ast}^\ast)}(l^\ast({\theta}))\circ l^{-1},$$
where $\mathcal{R}_{\us^\ast}^\ast$ is the Riesz measure defined on $\mathcal{P}_{\widetilde{\VV}}$ and for $m\in\mathcal{Q}_{\widetilde{\VV}}$, 
\begin{align*}
\V_{F(\mathcal{R}_{\us^\ast}^\ast)}(m)=
\pi\circ\left\{
\frac2{s_3}\rho(\hat m)\right.+ \left( \frac1{s_2}-\frac2{s_3}\right) 
\rho(\hat m-M_1) 
+\left. \left( \frac1{s_1}-\frac1{s_2}\right) 
\rho(\hat m-M_{1,2})\right\},
\end{align*}
with $M_1=\left[(\hat m^{-1})_{\{1:1\}}\right]^{-1}_0$ and $M_{1,2}=\left[(\hat m^{-1})_{\{1:2\}}\right]^{-1}_0$.
Recall that $\{1:i\}$ of $\mathcal{Z}_{\widetilde{\VV}}$, thus $\{1:1\}=\{1,2\}$ and $\{1:2\}=\{1,2,3\}$.
The projection $\pi\colon \mathrm{Sym}(4,\R)\to\mathcal{Z}_{\widetilde{\VV}}$ is given by
 $$\pi\colon\begin{pmatrix}x_{11} & x_{21} & x_{31} & x_{41} \\ x_{21} & x_{22} & x_{32} & x_{42} \\ x_{31} & x_{32} & x_{33} & x_{43} \\ x_{41} & x_{42} & x_{43} & x_{44} \end{pmatrix}
 \to \begin{pmatrix}\frac{x_{11}+x_{22}}{2} &  & x_{31} &  \\  & \frac{x_{11}+x_{22}}{2} & & x_{42} \\ x_{31} &  & x_{33} & \\  & x_{42} &  & x_{44} \end{pmatrix}.$$
For $m=\begin{pmatrix}m_3&&m_5&\\ &m_3&&m_4\\m_5&&m_2&\\&m_4&&m_1\end{pmatrix}$ we have
$
\hat m= 
\begin{pmatrix}
m_3-c&&m_5&\\ &m_3+c&&m_4\\
m_5&&m_2&\\
&m_4&&m_1
\end{pmatrix}$
with \mbox{$c=\frac12(m_4^2/m_1-m_5^2/m_2)$}.
Explicit formulas for matrices $M_1$ and $M_{1,2}$ can  be easily found. 
In particular,  for $\us=(p,p,p)$ one has
$$
\V_{F(\RV)}({\theta})=\frac1p(l^\ast)^{-1}\circ \pi\circ
\left[2\rho(\hat m) -  
\rho(\hat m-M_1)\right]
\circ l^{-1}
$$
where $m=l^\ast({\theta})$ and $\theta\in\CV$.
The last formula, compared with \eqref{pconstant}, confirms the fact that analysis of Wishart laws on homogeneous cones $\CV$ is technically more 
difficult than on the cones $\CVa$.
\end{example}
\section{Applications}\label{appli}

\subsection{Classical Wishart families $F(\mu_{{p}})$  on
$\mathrm{Sym}_+(n,\R)$}
  In this case, $\ZV=\mathrm{Sym}(r,\R)$,  $\CVa=\mathrm{Sym}_+(n,\R)$,  all $n_i=1$, 
  the projection $\pi$ is  the identity map from $\mathrm{Sym}(n,\R)$ to itself
  and $\hat{m}=m$.
  We have $\us={p}{\bf 1}$, 
the measure $\mu_{{p}}$ is the Riesz measure $\RVa$ and
$\us\in\mathfrak{X}$ if and only if ${p}\in\Lambda$.
 Formula \eqref{first} from the Introduction
 is instantly recovered using Theorem \ref{main}.

\subsection { Wishart families on symmetric cones, indexed by $\us\in\R^r$.}
 The cone $\mathrm{Sym}_+(n,\R)$ is the prime example of a symmetric cone,
 that is, a homogeneous cone $\Omega$ which is self-dual ($\Omega^\ast=\Omega$).

The matrix realization of homogeneous cones
(see Section  \ref{Matrix})
does
not coincide with the usual setting
in which symmetric cones are considered,
that is, Jordan algebras, except {for} $\mathrm{Sym}_+(n,\R)$
cone.
This is the reason why there is no ``automatic''
correspondence between formulas for variance functions
in {these} two settings. However, the techniques developed
in this article apply in the symmetric cone setting.

	Here we use the standard notation of \cite{FaKo1994}. Let $V$ be a simple Euclidean Jordan algebra
	of rank $r$ and let $\Omega$ be its associated irreducible symmetric cone. 
	{If $c$ is an idempotent in $V$, we denote by $V(c,1)$ the eigenspace corresponding to the eigenvalue $1$ of the linear operator $\mathbb{L}(x)$ on $V$, which is defined using the Jordan product $\mathbb{L}(x)=xy$.}
	For a fixed Jordan frame $\mathbf{c}=(c_1,\ldots,c_r)$ define subspaces $V^{(k)}=V(c_1+\ldots+c_k,1)$ and $W^{(k)}=V(c_{r-k+1}+\ldots+c_r,1)$. Denote by $P_k$ and $P_k^\ast$ the orthogonal projections of $V$ onto $V^{(k)}$ and $W^{(k)}$, respectively. \\	
	Let $\Delta_{\us}$ be the {generalized power function with
	respect to $\mathbf{c}$} and $\PP$ be the quadratic representation of $\Omega$. 
	We consider  natural exponential families generated by the Riesz
	measure $\RVa$ with the Laplace transform $\Delta_{-\us}$ and $\RV$
	with the Laplace transform $\theta\mapsto\Delta_{\us}(\theta^{-1})$.	
	Using the same techniques as in Theorem \ref{main}, but in
	the Euclidean Jordan algebra framework, we prove the following
\begin{proposition}
	For $m\in\Omega=\Omega^\ast$,
\begin{align*}
\V_{F(\RVa)}(m)=\frac{1}{s_1}\PP(m)+\sum_{i=2}^r \left(\frac1{s_{i}}-\frac1{s_{i-1}}\right) \PP\left(m-[P_{i-1}m^{-1}]^{-1}_{0}\right).
\end{align*}
\begin{align}\label{VSYM}
\V_{F(\RV)}(m)=\frac{1}{s_r}\PP(m)+\sum_{k=1}^{r-1} \left(\frac1{s_{k}}-\frac1{s_{k+1}}\right) \PP\left(m-[P^\ast_{r-k}m^{-1}]^{-1}_{0}\right).
\end{align}
\end{proposition}
Here $[\cdot]_0$ denotes the inclusion map from the subalgebras $V^{(k)}$ and $W^{(k)}$ to $V$.
\begin{remark}
Natural exponential families generated by the Riesz measure on symmetric cones were treated in \cite{HaLa01}.
{ In that paper a  formula 
(3.2), Th.3.2, p.935, for $\V_{F(\RV)}(m)$ is announced {but not proven}. 
It is different {and much more complicated} than \eqref{VSYM}
and its proof has gaps (on p.946 it is only proven that for {any} $\underline{S}$
there exists $\us$ such that the variance function of the NEF {generated by} ${\mathcal{R}}_{\underline{S}}$ is equal to the function $V_{\us}$ in (3.2),p.935. {In the present paper we have shown that actually $\underline{S}=\us$.} Note also misleading misprints
in \cite{HaLa01},
``Thm. 3.2" in place of ``Thm. 3.6",
in the title of Section 4  and on page 945.)}
\end{remark}

\subsection{Graphical homogeneous cones}
Let $G=(V,E)$ be an undirected graph, where $V=\{1,\ldots,r\}$ is the set of vertices and $E\subset V\times V$ is the set of undirected edges, that is, if $(i,j)\in E$ then $(j,i)\in E$, $i,j\in V$. 
For statisticians, the parameter space of interest for covariance graph models is the cone $P_G$ of positive definite matrices with fixed zeros corresponding to the missing edges of $G$. More precisely, if
$$\ZG:=\{ (x_{ij})\in\mathrm{Sym}(r,\R)\colon x_{ij}=0\mbox{ if }(i,j)\notin E\}$$
then $P_G$ is defined by
$$P_G:=\ZG\cap \mathrm{Sym}_+(r,\R).$$
It is known that the cone $P_G$ is homogeneous if and only if $G$ is decomposable (chordal) and does not contain the graph $\bullet-\bullet-\bullet-\bullet$, denoted by $A_4$, as an induced subgraph (for details see \cite{LM07,Is13}). For $1 \leq k < l \leq r$, we set $\VV_{lk}:=\R$ if $(k,l)\in E$, and $\VV_{lk}:=\{0\}$ otherwise. Then it can be shown that (possibly after renumeration of vertices), the family $\{\VV_{lk}\}_{1\leq k<l\leq r}$ satisfies \rm{(V1)-(V3)}. One sees ({\cite{Is13}}) that $\CV$ is a graphical cone if and only if $n_1=\ldots=n_r=1$.
{Thus the trace inner product coincides with the standard inner product.  }

 {The results of  the present paper apply to homogeneous graphical cones. However, in the present paper $\ZVa$ was identified with $\ZV$,
whereas in the statistical approach to graphical cones one proceeds as follows.}

Let $I_G$ be the real linear space of $G$-incomplete symmetric matrices, that is, functions $(i,j)\mapsto x_{ij}$ from $E$ to $\R$ such that $x_{ij} = x_{ji}$.
The dual space $\mathcal{Z}_G^\ast$ is identified with $I_G$ through 
$$\scalar{y,x}=\sum_{(i,j)\in E} x_{ij} y_{ij},\qquad (x,y)\in\ZG\times I_G$$
and the dual cone is denoted by
$Q_G:=\{ y\in I_G\colon \scalar{y,x}>0\,\,\forall x\in\overline{P_G}
\setminus\{0\}\}.$

Let $\pi\colon Z_G\to I_G$ be such that $\pi(x)_{ij}=x_{ij}$ for any $(i,j)\in E$.
For any $m\in Q_G$ there exists a unique $\hat{m}\in\mathrm{Sym}_+(r,\R)$ such 
that for all $(i,j)\in E$ one has $\hat{x}_{ij}=x_{ij}$ and such that
$\hat{x}^{-1}\in P_G$ (see \cite{LM07} p.1279).
The last definitions
of $\pi$ and $\hat{m}$ on graphical cones agree with the ones 
given in
 Definitions \ref{piDef} and \ref{hat}.
One {can} check (\cite{Is2016}) that the 
function ${Q_G\ni\eta\mapsto}H(\alpha,\beta,\eta)$ considered in \cite{LM07} equals to the {generalized power} function 
$\delta_{\us}$ for some $\us\in\R^r$, by comparing formula 
\eqref{New_delta}
with the definition of $H(\alpha,\beta,\eta)$.

Thus
 Theorem \ref{main} applies to the cone $Q_G\subset I_G$ with $n_i=1$.
 
 Similarly, by formula \eqref{Delta_delta}, the function
${P_G\ni y\mapsto} H(\alpha,\beta,\pi(y^{-1}))$ introduced in \cite{LM07} on the cone
$P_G$ coincide{s} with the {generalized power} function $\Delta_{\us}$ for some $\us$ and the results of 
Section \ref{VarP} apply to the cone $P_G$. 

\subsection{Non-homogeneous graphical cones}

Recently, the variance function was also computed for the cones $Q_G$, corresponding to 
non-homogeneous graphs $G=A_n$, $n\ge 4$, see
\cite{GIMamane}. The techniques are partly the same, but the lack of an analogue
of the equivariance formula \eqref{equiVar} must be overcome.\\


\noindent{\bf Appendix.} 
The unpublished paper \cite{BH09}
 may seem to contain results of our paper.
However,
 the main results of \cite{BH09}, announced in Theorems 3.3 and 4.2, 
 are false.
The proofs of these theorems are  based on several erroneous arguments.
Consequently, 
 the unpublished paper \cite{BH09} cannot be compared with the results of our paper. 
Here is a short explanation of failures of \cite{BH09}.\\

\noindent
[A] The decomposition (2.7) on Page 5 of \cite{BH09} is basic for their work,
 but it is not proven in the paper. 
It is false. 
Consider the following example.\\

Let $I=\{1,2,3,4\}$ and consider the following partial order on $I$:
$$
1\prec 3, 2\prec 3, 3\prec 4, 1 \not\prec \not\succ  2.
$$ 
This order generates a Vinberg algebra $\kA$ of $4\times 4$
matrices with $a_{12}=a_{21}=0$.

We find separator sets $S_1=S_2=\{3,4\}$ and the minimal elements set
$\wp=\{1,2\}$. 
Consider a diagonal matrix $$X=\mathrm{diag}(1,1,1,1).$$
Then, according to (2.6) of \cite{BH09}, 
$
X_1=\mathrm{diag}(1,0,0,-1), X_2=\mathrm{diag}(0,1,0,-1), X_3=\mathrm{diag}(0,0,1,1),\\
 X_4=\mathrm{diag}(0,0,0,1)
$
and $\sum_i X_i=\mathrm{diag}(1,1,1,0)\not=X$.\\

\noindent
[B] 
In Theorem 3.3 of \cite{BH09},
 the Riesz measure $R_\chi$
 is defined as a measure whose Laplace transform is $\Delta_\chi(\theta^{-1})$ for $\theta \in \mathcal{P}^\ast$,
 based on a wrong idea that, if $\theta = T^\ast\, T$ with $T \in \mathcal{T}^+_l$,
 then $\theta^{-1} = T^{-1} (T^{-1})^\ast$.
The equality does not hold in general because of non-associativity of Vinberg algebra,
 as is observed in the following example.\\

Define $\mathcal{A} := \set{A = (a_{ij}) \in \mathrm{Mat}(3,\R)}{a_{23} = a_{32} = 0}$,
 and let
 $\pi_{\mathcal{A}} : \mathrm{Mat}(3,\R) \to \mathcal{A}$ be
 the orthogonal projection with respect to the trace inner product of $\mathrm{Mat}(3,\R)$.
Then we have
$$
 \pi_{\mathcal{A}}(M) = 
 \begin{pmatrix} m_{11} & m_{12} & m_{13} \\ m_{21} & m_{22} & 0 \\ m_{31} & 0 & m_{33} \end{pmatrix}
 \quad (M = (m_{ij}) \in \mathrm{Mat}(3,\R)).
$$
We define a bilinear product on $\mathcal{A}$ by
$$
 A \cdot B := \pi_{\mathcal{A}}(AB) \in \mathcal{A} \qquad (A,B \in \mathcal{A}).
$$
This product together with the ordinary trace and the conjugation $A^\ast := A^{\top}$
 makes $\mathcal{A}$ a Vinberg algebra graded by a poset $I = \{1,2,3\}$ with $1 \prec 3$ and $2 \prec 3$.
Indeed, it is easy to check the axioms (i) -- (iv) in Page 3 of \cite{BH09}.
The axiom (v) is also easily verified
 because the space $\mathcal{T}_l = \set{\begin{pmatrix} t_{11} & 0 & 0 \\ t_{21} & t_{22} & 0 \\ t_{31} & 0 & t_{33} \end{pmatrix}}{t_{ij} \in \R}$
 is closed under the usual matrix product, so that $T \cdot U = TU$ for $T, U \in\mathcal{T}_l$. 
As for the axiom (vi), it suffices to check that 
\begin{equation} \label{eqn:axiom6} 
 \scalar{T \cdot (U \cdot U^\top),A} = \scalar{(T \cdot U) \cdot U^{\top}, A} \qquad (T,U \in \mathcal{T}_l,\,A \in \mathcal{A}). 
\end{equation}
Note that, for $A_1, A_2, A_3 \in \mathcal{A}$ in general, we have
$$ \scalar{A_1 \cdot A_2, A_3} = \mathrm{tr}\, (\pi(A_1A_2)A_3) = \mathrm{tr}\,A_1A_2 A_3. $$
If $A \in \mathcal{T}_l$, then the left-hand side of (\ref{eqn:axiom6}) is
$$
 \mathrm{tr}\,(T (U \cdot U^\top) A) = \mathrm{tr}\,( (U \cdot U^\top)AT) = \mathrm{tr}\, U U^\top AT,
$$
where the second equality above follows from $AT \in \mathcal{T}_l \subset \mathcal{A}$.
The right-hand side of (\ref{eqn:axiom6}) is
$$
 \mathrm{tr}\,((T \cdot U)U^\top A)
 =  \mathrm{tr}\,((TU)U^\top A) =  \mathrm{tr}\,TU U^\top A.
$$
Thus we obtain (\ref{eqn:axiom6}) in the case $A \in \mathcal{T}_l$. 
Similarly, we can show (\ref{eqn:axiom6}) for $A \in \mathcal{T}_u$.
Since every element of $\mathcal{A}$ is a sum of elements of $\mathcal{T}_l$ and $\mathcal{T}_u$, 
 we have (\ref{eqn:axiom6}) for $A \in \mathcal{A}$, and (vi) is verified.

The homogeneous cones $\mathcal{P}$ and $\mathcal{P}^\ast$ associated to the Vinberg algebra $\mathcal{A}$ are given by 
$$
 \mathcal{P} = \set{T \cdot T^\top}{T \in \mathcal{T}^+_l}, \quad
 \mathcal{P}^\ast = \set{T^\top \cdot T}{T \in \mathcal{T}^+_l}.
$$ 
The two cones are mutually dual in the vector space $\mathcal{H} = \set{X^\top = X}{X \in \mathcal{A}}$
 with respect to the trace inner product.
Note that, if $T \in \mathcal{T}_l$, then the matrix product $T^\top T$ belongs to $\mathcal{A}$.
Thus $T^\top \cdot T = T^\top T$.
Now we consider $T = \begin{pmatrix} 1 & 0 & 0 \\ 1 & 1 & 0 \\ 1 & 0 & 1 \end{pmatrix} \in \mathcal{T}^+_l$.
Then $\theta = T^\top \cdot T \in \mathcal{P}^\ast$ equals
$$\begin{pmatrix} 1 & 1 & 1 \\ 0 & 1 & 0 \\ 0 & 0 & 1 \end{pmatrix} 
\begin{pmatrix} 1 & 0 & 0 \\ 1 & 1 & 0 \\ 1 & 0 & 1 \end{pmatrix}
= \begin{pmatrix} 3 & 1 & 1 \\ 1 & 1 & 0 \\ 1 & 0 & 1 \end{pmatrix}.$$
On the other hand, put
$$ X := T^{-1} \cdot (T^\ast)^{-1} 
= \pi\left(\begin{pmatrix} 1 & 0 & 0 \\ -1 & 1 & 0 \\ -1 & 0 & 1 \end{pmatrix} 
  \begin{pmatrix} 1 & -1 & -1 \\ 0 & 1 & 0 \\ 0 & 0 & 1 \end{pmatrix} \right)
= \pi\begin{pmatrix} 1 & -1 & -1 \\ -1 & 2 & 1 \\ -1 & 1 & 2 \end{pmatrix}   
= \begin{pmatrix} 1 & -1 & -1 \\ -1 & 2 & 0 \\ -1 & 0 & 2 \end{pmatrix}.
$$
Then we observe
\begin{align*}
X \cdot \theta 
&= \pi\left( \begin{pmatrix} 1 & -1 & -1 \\ -1 & 2 & 0 \\ -1 & 0 & 2 \end{pmatrix}
 \begin{pmatrix} 3 & 1 & 1 \\ 1 & 1 & 0 \\ 1 & 0 & 1 \end{pmatrix}
 \right)
= \begin{pmatrix} 1 & 0 & 0 \\ -1 & 1 & 0 \\ -1 & 0 & 1 \end{pmatrix},\\
\theta \cdot X 
&= \pi\left( \begin{pmatrix} 3 & 1 & 1 \\ 1 & 1 & 0 \\ 1 & 0 & 1 \end{pmatrix}
\begin{pmatrix} 1 & -1 & -1 \\ -1 & 2 & 0 \\ -1 & 0 & 2 \end{pmatrix} \right)
= \begin{pmatrix} 1 & -1 & -1 \\ 0 & 1 & 0 \\ 0 & 0 & 1 \end{pmatrix}.
\end{align*}
Thus $X = T^{-1} \cdot (T^{-1})^*$ is not an inverse element of $\theta = T^* \cdot T$ in the Vinberg algebra $\mathcal{A}$.\\

\noindent 
[C] 
Keeping [B] in mind, we denote by $I(\theta)$ the element $T^{-1} \cdot (T^{-1})^\top \in \mathcal{P}$
 for $\theta = T^\top T \in \mathcal{P}^*$ with $T \in \mathcal{T}_l^+$,
 and the Riesz measure $R_\chi$ is a measure whose Laplace transform $L_{R_\chi}(\theta)$
 equals $\Delta_\chi(I(\theta)) = t_{11}^{-2\lambda_1} t_{22}^{- 2\lambda_2} t_{33}^{-2\lambda_3}$.
If $\lambda_1 = \lambda_2 = \lambda_3 = \lambda$, then $\Delta_\chi(I(\theta)) = (\det \theta)^{-\lambda}$,
 where $\det \theta$ stands for the ordinary determinant (not the Vinberg algebra determinant defined in Page 5 of \cite{BH09}).
In this case, Theorem 4.2 of \cite{BH09} implies
\begin{equation} \label{eqn:simple_VR}
 \V_{F(R_\chi)}(m) = \frac{1}{\lambda}P(m) \qquad (m \in \mathcal{P}),
\end{equation}
 where $P(m)A := m\cdot (A \cdot m)\,\,\,(A \in \mathcal{A})$,
 see the last sentence in Page 20 of \cite{BH09}. 
On the other hand,
 the mean map is given by $\mathcal{P}^* \owns \theta \mapsto m = \lambda I(\theta) \in \mathcal{P}$.
Let us examine $\V_{F(R_\chi)}(m_0)$ with $m_0 = \begin{pmatrix} 1 & -1 & -1 \\ -1 & 2 & 0 \\ -1 & 0 & 2 \end{pmatrix}$.
Thanks to the calculation in [B],
 the parameter corresponding to $m_0$ is $\theta_0 = \lambda \begin{pmatrix} 3 & 1 & 1 \\ 1 & 1 & 0 \\ 1 & 0 & 1 \end{pmatrix}$.
Then $\scalar{\V_{F(R_\chi)}(m_0) I_3, I_3}$ is the second derivative of $- \lambda \log \det \theta$ at $\theta = \theta_0$ in the direction of $I_3$,
 so that
 $$ \scalar{\V_{F(R_\chi)}(m_0) I_3, I_3} = \lambda\, \mathrm{tr}\, (\theta_0^{-1} I_3 \theta_0^{-1} I_3)
 = \frac{1}{\lambda}\, \mathrm{tr}  \begin{pmatrix} 1 & -1 & -1 \\ -1 & 2 & 1 \\ -1 & 1 & 2 \end{pmatrix}^2
 = \frac{15}{\lambda}. $$
However, we have
 \begin{align*}
 \frac{1}{\lambda} \scalar{P(m_0)I_3, I_3} = \frac{1}{\lambda} \scalar{(m_0 \cdot (I_3 \cdot m_0)), I_3} 
 = \frac{1}{\lambda} \scalar{m _0\cdot m_0, I_3} = \frac{1}{\lambda} \mathrm{tr}\,m_0^2 = \frac{13}{\lambda}.
 \end{align*}
Therefore (\ref{eqn:simple_VR}) and Theorem 4.2 of   \cite{BH09} are false.

Let us recall the space $\ZV$ in Example \ref{Vin}.
We have a linear isomorphism
 $$ \ell : \ZV \owns X = \begin{pmatrix} x_{11} & 0 & x_{31} \\ 0 & x_{22} & x_{32} \\ x_{31} & x_{32} & x_{33} \end{pmatrix}
 \mapsto \rho(W)X = \begin{pmatrix} x_{33} & x_{32} & x_{31} \\ x_{32} & x_{22} & 0 \\ x_{31} & 0 & x_{11} \end{pmatrix} \in \mathcal{H} = \mathcal{A} \cap \mathrm{Sym}(3,\R), $$
 where $W = \begin{pmatrix} 0 & 0 & 1 \\ 0 & 1 & 0 \\ 1 & 0 & 0 \end{pmatrix}$. 
Then we have
 $$ \ell(\CV) = \set{(WTW^{-1}) (WTW^{-1})^\top }{T \in H_{\mathcal{V}}}
 = \set{U^\top U}{U \in \mathcal{T}_l^+} = \mathcal{P}^\ast,$$
 so that $\ell(\CVa) = \mathcal{P}$. 
Using the isomorphism $\ell$, we can deduce from (\ref{pconstant}) the correct formula $\V_{F(R_\chi)}(m) = \frac{1}{\lambda} \pi_{\mathcal{A}} \circ \rho(\hat{m})$,
 where $\hat{m} = \theta^{-1} \in \mathrm{Sym}(3,\R)$.

\bibliographystyle{plainnat}


\def\cprime{$'$}

\end{document}